\documentclass[10pt]{amsart}
\usepackage{amssymb, enumitem}
\usepackage[all]{xy}
\usepackage{hyperref, aliascnt}
\usepackage{mathtools}
\usepackage[english]{babel}
\usepackage{tikz, tikz-cd}
\usetikzlibrary{matrix,decorations.pathmorphing,arrows}
\tikzset{commutative diagrams/.cd,arrow style=tikz,diagrams={>=stealth'}}
\newcounter{TmpEnumi}

\newcommand{\Z}{{\mathbb{Z}}}

\newcommand{\N}{{\mathbb{N}}}

\newcommand{\W}{{\mathcal{W}}}

\newcommand{\id}{{\mathrm{id}}}

\newcommand{\card}{{\mathrm{card}}}
\newcommand{\Aut}{{\mathrm{Aut}}}

\theoremstyle{definition}
\newtheorem{lma}{Lemma}[section]

\numberwithin{equation}{section}

\newaliascnt{thmCt}{lma}
\newtheorem{thm}[thmCt]{Theorem}
\aliascntresetthe{thmCt}

\newaliascnt{corCt}{lma}
\newtheorem{cor}[corCt]{Corollary}
\aliascntresetthe{corCt}

\newaliascnt{propCt}{lma}
\newtheorem{prop}[propCt]{Proposition}
\aliascntresetthe{propCt}

\newtheorem*{thm*}{Theorem}
\newtheorem*{qst*}{Question}
\newtheorem*{cor*}{Corollary}
\newtheorem*{prop*}{Proposition}

\newaliascnt{pgrCt}{lma}

\aliascntresetthe{pgrCt}

\newaliascnt{dfCt}{lma}
\newtheorem{df}[dfCt]{Definition}
\aliascntresetthe{dfCt}

\newaliascnt{remCt}{lma}
\newtheorem{rem}[remCt]{Remark}
\aliascntresetthe{remCt}

\newaliascnt{remsCt}{lma}

\aliascntresetthe{remsCt}

\newaliascnt{egCt}{lma}
\newtheorem{eg}[egCt]{Example}
\aliascntresetthe{egCt}

\newaliascnt{egsCt}{lma}

\aliascntresetthe{egsCt}

\newaliascnt{qstCt}{lma}

\aliascntresetthe{qstCt}

\newaliascnt{pbmCt}{lma}
\newtheorem{pbm}[pbmCt]{Problem}
\aliascntresetthe{pbmCt}

\newaliascnt{notaCt}{lma}
\newtheorem{nota}[notaCt]{Notation}
\aliascntresetthe{notaCt}

\AtBeginDocument{   \def\MR#1{}
}

\begin{document}
\title[Weak tracial Rokhlin property]{The weak tracial Rokhlin property for 
finite group actions on simple C*-algebras}
\author[M. Forough]{Marzieh Forough}
\address[\textbf{Marzieh Forough}]{
Current address: Institute of mathematics, 
Czech academy of sciences, 115 67 Praha 1, Czech republic;
Previous Address: School of Mathematics,
 Institute for Research in Fundamental
 Sciences (IPM), P. O.\  Box 19395--5746, Tehran, Iran}
\email{forough@math.cas.cz, mforough86@gamil.com}
\author[N. Golestani]{Nasser Golestani}
\address[\textbf{Nasser Golestani}]{Department of Pure Mathematics, Faculty of Mathematical Sciences,
   Tarbiat Modares University, P.O. Box 14115--134,
 Tehran, Iran; School of Mathematics,
 Institute for Research in Fundamental
 Sciences (IPM), P. O.\  Box 19395--5746, Tehran, Iran}
\email{n.golestani@modares.ac.ir}
\dedicatory{}

\begin{abstract}
We develop the concept of weak tracial Rokhlin property for 
finite group actions on simple (not necessarily unital) C*-algebras
and study its properties  systematically.
In particular, we show that this property is stable under restriction to  
invariant hereditary C*-algebras,  minimal tensor products, and direct limits of actions.
Some of these results are new even in the unital case and answer
open questions asked by N.~C.~Phillips in full generality.
We present several examples of finite group actions with the 
weak tracial Rokhlin property on simple 
stably projectionless C*-algebras. 
We prove that if $\alpha \colon G \rightarrow \mathrm{Aut}(A)$ is an action of a 
finite group $G$ on a simple   C*-algebra 
$A$ with tracial rank zero 
and $\alpha$ has the weak tracial Rokhlin property, 
then the crossed product $A \rtimes _{\alpha} G$ and the fixed point algebra 
$A^{\alpha}$ are  simple  with tracial rank zero. This extends a result of 
N.~C.~Phillips to the nonunital case. 
We use the machinery of Cuntz subequivalence to work in this nonunital setting.
\end{abstract}

\maketitle
\tableofcontents

\section{Introduction}
The Rokhlin property for actions on C*-algebras appeared in 
 \cite{HO84, BSKR93, Ki95, Ki96}.
Izumi   gave a modern definition of the Rokhlin property for finite group actions on 
unital  C*-algebras \cite{IZ04a, IZ04b}. 
This property is useful to understand the structure of
the crossed product of C*-algebras and properties passing from the underlying algebra to
the crossed product \cite{OP12}.
However, actions with the Rokhlin property are rare and 
 many C*-algebras admit no finite group actions with the Rokhlin property. 
 Indeed, the Rokhlin property imposes severe K-theoretical obstructions on C*-algebras. 
   Phillips introduced  the tracial   Rokhlin property 
 for finite group actions on simple unital C*-algebras \cite{Ph11}
  with the purpose of proving that every simple higher dimensional
noncommutative torus is an AT~algebra \cite{Ph06}, and proving that certain
crossed products of such algebras by finite cyclic groups are AF~algebras
\cite{Echterhoff-Phillips}  
 (see \cite{OP06} for $\mathbb{Z}$ actions with this property).
 The tracial Rokhlin property is generic  in many cases 
 (see \cite{Ph12} and   \cite[Chapter~4]{Wang13}),
 and also can be used to study properties passing from the 
 underlying algebra to the 
 crossed product  \cite{Ph11, Echterhoff-Phillips, Archey}.
 
 Weak versions of the tracial Rokhlin property 
 in which
 one uses orthogonal positive contractions instead of 
 orthogonal projections were studied for actions
 on simple unital C*-algebras with few projections
 \cite{Ph12, MS12, Hirshberg, Wang13, GHS, Wang18}
 (see Definition~\ref{unitalgar}).
As an example, the flip action on the Jiang-Su algebra
$\mathcal{Z} \cong \mathcal{Z} \otimes \mathcal{Z}$
has the weak tracial Rokhlin property but it does not have
the tracial Rokhlin property 
\cite{Hirshberg}. 
 
The Rokhlin property was extended to the case of actions on nonunital 
C*-algebras   \cite{Nawata, Sa15, GS16},
and there are actions with the Rokhlin property on stably projectionless
C*-algebras, in particular on the Razak-Jacelon algebra $\mathcal{W}$
 \cite{Nawata}.
However,
 there has been  no work on extending the (weak) tracial Rokhlin property to 
 the simple  nonunital case.
(As far as we know, a suitable definition of the   tracial Rokhlin property for actions on
nonsimple C*-algebras is not known.)
Moreover, actions on simple nonunital C*-algebras naturally appear, for instance,
the restriction of an action on a simple unital C*-algebra to an invariant nonunital
hereditary subalgebra (see Proposition~\ref{propwtrpher}).
Also, there are many  examples of finite group actions on simple nonunital C*-algebras 
without the Rokhlin property, which have the weak tracial Rokhlin property
(see 
Example~\ref{example_all}). 
In fact, the problem of finding the right definition of 
the  tracial Rokhlin property for actions on simple nonunital
C*-algebras was asked by Phillips \cite{Ph.free}.
  This motivated us to investigate the weak tracial Rokhlin property for 
  finite group actions on simple  C*-algebras. 
  We give the following definition:
\begin{df}\label{weak.t.R}
Let $\alpha \colon G \rightarrow \Aut (A)$ be an action of a finite group $G$ on a 
simple  C*-algebra $A$. 
We say that $\alpha$ has the  \emph{weak tracial Rokhlin property} if 
for every $\varepsilon > 0$,  
every finite subset $F \subseteq A$, and all positive elements $x,y \in A$ with $\|x\|=1$, 
there exists a family of orthogonal positive contractions 
$(f_{g})_{g \in G}$ in $A$ such that, with $f=\sum_{g \in G} f_{g}$,
  the following hold:
\begin{enumerate}
\item\label{weak.t.R_it1} 
$\|f_{g}a-af_{g}\| < \varepsilon$ for all $a\in F$ and all $g\in G$;
\item\label{weak.t.R_it2}  
$\|\alpha_{g}(f_{h})-f_{gh}\| < \varepsilon$ for all $g,h\in G$;
\item\label{weak.t.R_it3} 
$(y^{2}-yfy - \varepsilon)_{+} \precsim_{A} x$;
\item\label{weak.t.R_it4} 
$\|fxf\| > 1-\varepsilon$.
\end{enumerate}
We say that  $\alpha$ has the  \emph{tracial Rokhlin property} if 
we can  arrange $(f_{g})_{g \in G}$ above
 to be mutually orthogonal projections.
\end{df}

 It turns out that 
our definition of  the tracial Rokhlin property
extends Phillips's definition of the tracial Rokhlin property \cite[Definition~1.2]{Ph11} 
to the nonunital simple case. 
We recall that in Phillips's definition of the tracial Rokhlin property 
for  finite group actions on   simple unital C*-algebras, 
Condition~\eqref{weak.t.R_it3} is formulated as follows:
\begin{enumerate}
\item[$(3)'$]\label{t.R_it3} 
$1-f \precsim_{A} x$ (or equivalently, $1-f$ is Murray von~Neumann 
equivalent to a projection in the hereditary subalgebra generated by $x$),
\end{enumerate}
where $(f_{g})_{g \in G}$ is a family of orthogonal projections in $A$.
Phillips in \cite{Ph.free} asked for a correct analogue of Condition~$(3)'$
 in the  simple nonunital case. 
Condition~\eqref{weak.t.R_it3} in
Definition~\ref{weak.t.R} contains our main idea 
for a suitable notion of the weak tracial Rokhlin property (as well as tracial Rokhlin property) in the nonunital case. 
This condition---which may seem
strange at the first glance---says that $1-f$ is small 
with respect to the Cuntz subequivalence relation.
The rationale behind this condition is that since $y\in A_{+}$ is arbitrary, we can take it to be arbitrarily
large (that is, close to $1$) and so 
$y^{2}-yfy=y(1-f)y$ is close to $1-f$. The $\varepsilon$ gap in this condition is a technical condition
 needed, for example, when applying a key lemma in the Cuntz semigroup 
(Lemma~\ref{lemkr}). 

The following result can be considered as a generalization of 
\cite[Theorem~1.9]{Ph15} to the nonunital case
(see Theorem~\ref{wtr=tr}).

\begin{thm}\label{thm_tr0_intro}
Let $A$ be a simple  C*-algebra with tracial rank zero, 
and let $\alpha \colon G \rightarrow \Aut(A)$ be an action of a finite group $G$ on $A$. 
If $\alpha$ has the weak tracial Rokhlin property 
then it has the  tracial Rokhlin property.
\end{thm}

Phillips in \cite[Problem~3.2]{Ph.free} asked whether there is a reasonable formulation 
of the tracial Rokhlin property for finite group actions on simple unital  C*-algebras
in terms of the central sequence algebra. 
We give an answer to this question in the \emph{not necessarily unital} simple case.
Indeed, it turns out that if moreover $A$ is separable and one works
with the central sequence algebra $A_{\infty}$, then Condition~\eqref{weak.t.R_it3}
can be replaced by 
$y^{2}-yfy\precsim_{A_{\infty}} x$
(see Proposition~\ref{propcs}).

We prove that an action with the weak tracial Rokhlin property 
is pointwise outer (Proposition~\ref{outer}), and hence the resulting crossed product is simple. 
Moreover, we prove several permanence properties for finite group actions with
 the weak tracial Rokhlin property on simple  C*-algebras, 
 for example,  
  passing to restriction to  
invariant hereditary C*-algebras,  minimal tensor products, and direct limits of actions. 
In particular, the following result concerning tensor products gives an affirmative answer to 
a question of Phillips
 \cite[Problem~3.18]{Ph.free} (see Theorems~\ref{proptensor} and \ref{thmtensor}).

\begin{thm}\label{thm_ex_intro} 
 Let $\alpha \colon G\to \mathrm{Aut}(A)$ and $\beta \colon G\to \mathrm{Aut}(B)$
be actions of a finite group $G$ on simple C*-algebras $A$ and $B$.
If $\alpha$ has the weak tracial Rokhlin property
then so does $\alpha \otimes \beta \colon G \to \mathrm{Aut}(A \otimes_{\min} B)$.
If $\alpha$ has the  tracial Rokhlin property
then so does $\alpha \otimes \beta \colon G \to \mathrm{Aut}(A \otimes_{\min} B)$
whenever $B^{\beta}$ has an
approximate identity (not necessarily increasing) consisting of projections. 
\end{thm}

 Phillips proved that the crossed product of a simple unital C*-algebras with tracial rank zero 
 by a finite group action with the tracial Rokhlin property, is again simple with tracial rank zero \cite{Ph11}. 
The following   theorem generalizes this result to the nonunital case
(see Theorem~\ref{thmcross}).

\begin{thm}\label{thm_cross_wtrp_intro}
Let $A$ be a simple    C*-algebra with tracial rank zero and let $\alpha$
be an action of a finite group $G$ on $A$ with the weak tracial Rokhlin property.
Then the crossed product  $A \rtimes _{\alpha} G$ and the fixed point algebra
$A^{\alpha}$ are simple C*-algebras with tracial rank zero.
\end{thm}

The preservation of some
other classes of simple C*-algebras under taking crossed products
by finite group actions with the (weak) tracial Rokhlin property
is given in \autoref{sec_cross} (and in 
\cite{AGJP17} and \cite{G19}).

To prove Theorem~\ref{thm_cross_wtrp_intro},
we need to work with simple nonunital C*-algebras with tracial rank zero.
Recall that Lin in \cite{Lin01} first gave the definition of   tracial rank for unital C*-algebras 
and then he defined the tracial rank of a nonunital C*-algebra to be the tracial  rank of its minimal 
unitization. However, working with the unitization of C*-algebras is not always convenient.  
Moreover, the unitization of a simple nonunital C*-algebra is not simple and so one can not use  
techniques which are applicable 
only to simple C*-algebras.  To deal with this difficulty, we develop an approach which 
unifies the concept of tracial rank
zero for both unital and nonunital simple C*-algebras; see Theorem~\ref{main.trk}. 
This approach helps us to study  crossed products of  simple nonunital C*-algebras with tracial rank zero 
by finite group actions with the weak tracial Rokhlin property. We also need some results
about simple nonunital C*-algebras with tracial rank zero, such as Morita invariance
and having real rank zero and stable rank one
(the last two results in the simple \emph{unital} case are proved in \cite{Lin01}). 
We did not find any reference proving these
results (in the nonunital case), however, they may be known to some researchers.
So we  prove them in \autoref{app}.




\section{Cuntz subequivalence}\label{secpre}
In this section, we recall some results on  Cuntz subequivalence 
and provide some lemmas which will be used in the subsequent sections.
We refer the reader to \cite{APT11} and \cite{Ph14} for more information
about  Cuntz subequivalence.
\begin{nota}
We use the following notation  in this paper. 
\begin{enumerate}
\item 
For a C*-algebra $A$, $A_{+}$ denotes the positive cone of $A$. Also,
$A^{+}$ denotes the forced  unitization of $A$ 
(adding a  new identity  even if  $A$  is unital),
while $A^{\sim} = A$ if $A$ is unital and
$A^{\sim} = A^{+}$ if $A$ is nonunital.
\item
If $p $ and $q$ are projections in a C*-algebra $A$, then we write 
$p\sim_{\mathrm{MvN}}q$ if $p$ is Murray-von Neumann equivalent to $q$.

\item
If $E$ and $F$ are  subsets of a C*-algebra $A$ and $\varepsilon>0$, 
then we write $E\subseteq_{\varepsilon} F$ if
for every $a\in E$ there is $b\in F$ such that $\|a-b\|<\varepsilon$.

\item
We write $\mathcal{K}=K(\ell^{2})$ and
$M_{n}=M_{n}(\mathbb{C})$.
 
 \item
 Let $A$ be a C*-algebra. For $a,b\in A_{+}$, we say that
 $a$ is \emph{Cuntz subequivalent} to
$b$ in $A$ and we write $a\precsim_{A} b$,
 if there is a sequence $(v_n)_{n \in \mathbb{N}}$ in $A$ 
 such that $\|a-v_{n}bv_{n}^{*}\|\to 0$.
We write $a\sim_{A} b$ if both $a\precsim_{A} b$ 
and $b\precsim_{A} a$. 
 If $a,b\in (A\otimes \mathcal{K})_{+}$, we write 
 $a\precsim_{A} b$ if $a$ is {Cuntz subequivalent} to
$b$ in $A\otimes \mathcal{K}$.
$[a]$ stands for the
Cuntz equivalence class of $a$.

 \item
 Let $a$ be a positive element in a C*-algebra $A$
 and let $\varepsilon>0$. Let
 $f_{\varepsilon}\colon [0,\infty) \to [0,\infty)$ be defined by
 $f_{\varepsilon}=0$ on $[0,\varepsilon]$ and
 $f_{\varepsilon}(\lambda)=\lambda-\varepsilon$ on
 $(\varepsilon,\infty)$. We denote 
 $(a-\varepsilon)_+=f_{\varepsilon}(a)$.

 \item
 We use the notation $\mathbb{Z}_{n}$ for the group
 $\mathbb{Z}/n\mathbb{Z}$. Moreover,
 $\mathbb{N}$ denotes the set of natural
 numbers not including zero.
\end{enumerate}
\end{nota}

The following key lemma  will be used several times throughout the paper.

\begin{lma}[\cite{KR02}, Lemma~2.2]\label{lemkr}
Let $A$ be a C*-algebra, let $a,b\in A_{+}$, and let $\varepsilon >0$. If
$\|a-b\|<\varepsilon$ then there is a contraction $d\in A$ such that
$(a-\varepsilon)_{+}=dbd^{*}$. In particular, $(a-\varepsilon)_{+}\precsim_{A} b$.
\end{lma}

In the preceding lemma, if instead of $\|a-b\|<\varepsilon$ we assume
that $\|a-b\|\leq\varepsilon$, then again we get $(a-\varepsilon)_{+}\precsim_{A} b$.
In fact, for any $\delta>0$ we have $\|a-b\|< \varepsilon + \delta$ and so
$(a-\varepsilon-\delta)_{+}\precsim_{A} b$. Letting $\delta \to 0$, we get
$(a-\varepsilon)_{+}\precsim_{A} b$.

We need the following lemma   to work with Condition~\eqref{w.t.R.p_it3} 
in Definition~\ref{w.t.R.p}.

\begin{lma}[\cite{AGJP17}]\label{lemmkey}
Let $A$ be a C*-algebra, let $x\in A $ be a nonzero element, and let $b\in A_{+}$. Then for any $\varepsilon>0$,
\[
(xbx^{*}-\varepsilon)_{+}\precsim_{A} x \left(b- {\varepsilon}/{\|x\|^{2}}\right)_{+}x^{*}.
\]
In particular, if $\|x\|\leq 1$ then
$(xbx^{*}-\varepsilon)_{+}\precsim_{A} x (b-\varepsilon)_{+}x^{*}\precsim_{A} (b-\varepsilon)_{+}$.
\end{lma}

\begin{proof}
	We have 
	\begin{align*}
	\| xbx^{*} - x \left( b -  {\varepsilon}/{\|x\|^{2}}\right)_{+} x^{*}\| \leq \|x\|^{2}\| b - 
	\left( b - {\varepsilon}/{\|x\|^{2}}\right)_{+} \| \leq \|x\|^{2}\frac{\varepsilon}{\|x\|^{2}} = \varepsilon.
	\end{align*}
Using the remark following 
Lemma~\ref{lemkr}, we get
\[
		(xbx^{*}-\varepsilon)_{+}\precsim_{A} x \left(b- {\varepsilon}/{\|x\|^{2}}\right)_{+}x^{*}.
\]
If $ \|x\|\leq 1 $ then $ \tfrac{\varepsilon}{\|x\|^{2}} \geq \varepsilon $ and so 
$(xbx^{*}-\varepsilon)_{+}\precsim_{A} x (b-\varepsilon)_{+}x^{*} \precsim_{A} (b - \varepsilon)_{+}$.
	\end{proof}

The following lemma is well-know. (It follows from
 \cite[Lemmas~2.2 and 2.4(i)]{KR02}.)

\begin{lma}\label{lembdd1}
Let $A$ be a C*-algebra, let $a,b\in A_{+}$, and let $\delta>0$. If
$a\precsim_{A} (b-\delta)_{+}$ then there exists a bounded sequence $(v_{n})$
in $A$ such that $\|a-v_{n}bv_{n}^{*}\|\to 0$. We can take this sequence such that
$\|v_{n}\|\leq \|a\|^{\frac{1}{2}}\delta^{-\frac{1}{2}}$ for every $n\in \mathbb{N}$.
\end{lma}

The following lemma is known and follows from \cite[Lemmas~2.18 and 2.19]{APT11}.

\begin{lma}\label{lemmvn}
Let $A$ be a C*-algebra, let $a\in A_{+}$, and let $p\in A$ be a projection. 
The following statements are equivalent:
\begin{enumerate}

\item\label{lemmvn_it1} 
$p\precsim_{A} a$;
\item\label{lemmvn_it2}  
there exists $v\in A$ such that $p=vav^{*}$;
\item\label{lemmvn_it3}  
$p\sim_{\mathrm{MvN}} q$ in $A$ for some projection $q$ in $\overline{aAa}$.

\end{enumerate}
\end{lma}

Note that the statements in Lemma~\ref{lemmvn} are also equivalent
to $[p]\leq [a]$ in the sense of 
\cite[Definition~2.2]{Lin01}
(see \autoref{app}).

In general, there is no upper bound for the norm of $v$ in the previous lemma, unless
 there is a gap between $p$ and $a$; see the following lemma
(which may be considered as a special case of \cite[Lemma~2.4]{KR02}).

\begin{lma}\label{lembddp}
Let $A$ be a C*-algebra, let $a\in A_{+}$, let $\varepsilon>0$, and let $p\in A$ be a projection. 
If $p\precsim_{A} (a-\varepsilon)_{+}$, then
there exists $v\in A$ such that $p=vav^{*}$ and $\|v\|\leq \varepsilon^{-\frac{1}{2}}$.
\end{lma}

\begin{proof}
By  Lemma~\ref{lemmvn},  there exists $w\in A$ such that 
 $p=w(a-\varepsilon)_{+}w^{*}$. Then  \cite[Lemma~2.4(i)]{KR02} implies that
 there is $v\in A$ such that $p=vav^{*}$ and $\|v\|\leq \varepsilon^{-\frac{1}{2}}$.
\end{proof}
 
Part~\eqref{lembdd_it1} of the following lemma is a variant of \cite[Lemma~2.4]{KR02}. 
We shall use this lemma in the proof of 
 Lemma~\ref{lemtrp2}.
\begin{lma}\label{lembdd}
Let $A$ be a C*-algebra, let $a,b\in A_{+}$, and let $\varepsilon>0$. 
\begin{enumerate}

\item\label{lembdd_it1}
If $a=x(b-\varepsilon)_{+}$ for some $x\in A$, then $a=yb$ for some $y\in A$ with
$\|y\|\leq \varepsilon^{-1}\|a\|$.
\item\label{lembdd_it2} 
If $a\in \overline{A(b-\varepsilon)_{+}}$ then there is a sequence 
$(v_{n})$ in $A$ such that
$\|a-v_{n}b\|\to 0$ and $\|v_{n}\|\leq \varepsilon^{-1}(\|a\|+\frac{1}{n})$ 
for all $n\in \mathbb{N}$.
\end{enumerate}
\end{lma}

\begin{proof}
We define  continuous functions $f_{\varepsilon},g_{\varepsilon}\colon [0,\infty)\to [0,\infty)$ for $\varepsilon>0$ as in \cite[Lemma~2.4]{KR02},
that is,
\[
f_{\varepsilon}(t)=
\begin{cases}
\sqrt{\dfrac{t-\varepsilon}{t}}  &t\geq \varepsilon\\
0    &t<\varepsilon
\end{cases}
\ \ \ \text{and}\ \ \ 
g_{\varepsilon}(t)=
\begin{cases}
\dfrac{1}{t}  &t\geq \varepsilon\\
\varepsilon^{-2}t    &t<\varepsilon.
\end{cases}
\]
Then $tf_{\varepsilon}(t)^{2}=(t-\varepsilon)_{+}$ and 
$f_{\varepsilon}(t)^{2}=(t-\varepsilon)_{+}g_{\varepsilon}(t)$. Thus
$bf_{\varepsilon}(b)^{2}=(b-\varepsilon)_{+}$ and 
$f_{\varepsilon}(b)^{2}=(b-\varepsilon)_{+}g_{\varepsilon}(b)$.
Note that $\|g_{\varepsilon}(b)\|\leq \varepsilon^{-1}$.

To prove \eqref{lembdd_it1}, put $y=xf_{\varepsilon}(b)^{2}$. Then
$yb= xf_{\varepsilon}(b)^{2}b=x(b-\varepsilon)_{+}=a$. Also,
\[
yy^{*}=xf_{\varepsilon}(b)^{4}x^{*}=x(b-\varepsilon)_{+}^{2}g_{\varepsilon}(b)^{2}x^{*}
\leq \|g_{\varepsilon}(b)^{2}\|x(b-\varepsilon)_{+}^{2}x^{*}\leq \varepsilon^{-2}aa^{*}.
\]
Thus $\|y\|\leq \varepsilon^{-1}\|a\|$.

For \eqref{lembdd_it2}, let $a\in \overline{A(b-\varepsilon)_{+}}$ 
and {f}ix $n\in \mathbb{N}$. Then there is $w_{n}\in A$
such that $\|a-w_{n}(b-\varepsilon)_{+}\|<\frac{1}{n}$. Put $a_{n}=w_{n}(b-\varepsilon)_{+}$.
Thus $\|a_{n}\|\leq \|a\|+\frac{1}{n}$. By \eqref{lembdd_it1} there is $v_{n}\in A$ such that
$a_{n}=v_{n}b$ and 
$\|v_{n}\|\leq \varepsilon^{-1}(\|a\|+\frac{1}{n})$. Then
$\|a-v_{n}b\|=\|a-a_{n}\|<\frac{1}{n}$, and so $\|a-v_{n}b\|\to 0$.
\end{proof}

\section{The weak tracial Rokhlin property}\label{sec_wtRp}


In this section, we define the weak tracial Rokhlin property
(as well as the tracial Rokhlin property) for 
finite group actions on simple not necessarily unital C*-algebras. 
We show that the weak 
tracial Rokhlin property implies pointwise outerness and so the resulting 
crossed product is simple. 
Then we compare it with other notions of the weak tracial Rokhlin property 
for actions on simple unital C*-algebras. 
Moreover, we show that the Rokhlin property
in the sense of \cite[Definition~3.2]{Sa15}  implies the weak tracial Rokhlin property
for actions on simple C*-algebras.

\begin{df}\label{w.t.R.p}
Let $\alpha \colon G \rightarrow \Aut (A)$ be an action of a finite group $G$ on a 
simple  C*-algebra $A$. 
We say that $\alpha$ has the  \emph{weak tracial Rokhlin property} if 
for every $\varepsilon > 0$,  
every finite subset $F \subseteq A$, and all positive elements $x,y \in A$ with $\|x\|=1$, 
there exists a family of orthogonal positive contractions 
$(f_{g})_{g \in G}$ in $A$ such that, with $f=\sum_{g \in G} f_{g}$,
  the following hold:
\begin{enumerate}
\item\label{w.t.R.p_it1}
 $\|f_{g}a-af_{g}\| < \varepsilon$ for all $a\in F$ and all $g\in G$;
\item\label{w.t.R.p_it2}
 $\|\alpha_{g}(f_{h})-f_{gh}\| < \varepsilon$ for all $g,h\in G$;
\item\label{w.t.R.p_it3}
  $(y^{2}-yfy - \varepsilon)_{+} \precsim_{A} x$;
\item\label{w.t.R.p_it4}
  $\|fxf\| > 1-\varepsilon$.
\end{enumerate}
We say that  $\alpha$ has the  \emph{tracial Rokhlin property} if 
we can  arrange $(f_{g})_{g \in G}$ above
 to be mutually orthogonal projections.
\end{df}

An action  $\alpha \colon G \rightarrow \mathrm{Aut}(A)$ is called \emph{pointwise outer}
 if for any $ g \in G\setminus \{1\}$, the automorphism $\alpha_{g}$ is outer, that is, 
 it is not of the form 
 ${\mathrm{Ad}} (u)$ 
 for any unitary $u$ in the multiplier algebra of $A$. 
 
 \begin{prop} \label{outer}
Let $\alpha$ be an action of a finite group $G$ on a simple C*-algebra $A$. 
If $\alpha$ has the weak tracial Rokhlin property then $\alpha$ is pointwise outer.
\end{prop}

\begin{proof}
The idea of the proof is similar to that of 
\cite[Proposition~3.2]{Sa15}. However, we  
 need  Condition~\eqref{w.t.R.p_it4} in 
Definition~\ref{w.t.R.p} 
instead of Condition~(iii) in 
\cite[Definition~3.2]{Sa15}, and so we need more estimates.
 Suppose to the contrary that
there are $g_{0}\in G\setminus\{1\}$ and a unitary $u $ 
in the multiplier algebra of $A$ such that $\alpha_{g_{0}}={\mathrm{Ad}} (u)$.
Set $n=\card(G)$.
Choose $\varepsilon$ with $0<\varepsilon<1$ such that
\[
\dfrac{\sqrt{1-\varepsilon}-n\varepsilon}{n}>0 \ \ \text{and}\ \
\left(\dfrac{\sqrt{1-\varepsilon}-n\varepsilon}{n}\right)^{2}-4\varepsilon>0.
\]
 By \cite[Lemma~2.5.11]{Lin.book} (with $f(t)=t^{\frac{1}{2}}$ there), 
there is $\delta>0$ such that if $x,y\in A$ are positive contractions with
$\|xy-yx\|<\delta$ then 
$\|x^{\frac{1}{2}}y-yx^{\frac{1}{2}}\|<\varepsilon$. We may assume that $\delta<\varepsilon$.
Choose a positive element $b\in A^{\alpha}$ with $\|b\|=1$. 
Applying Definition~\ref{w.t.R.p} with $F=\{b, bu^{*}\}$, with $\delta$ in place of $\varepsilon$,
with $b$ in place of $x$, and with $y=0$, we obtain a family of
orthogonal positive contractions $(f_{g})_{g\in G}$ in $A$ such that
\begin{enumerate}

\item\label{outer_it1}
$\|f_{g}b-bf_{g}\|<\delta$ and $\|f_{g}bu^{*}-bu^{*}f_{g}\|<\delta$ for all $g\in G$;
\item\label{outer_it2}
$\|\alpha_{g}(f_{h})-f_{gh}\|<\delta$ for all $g,h\in G$;
\item\label{outer_it3}
$\|fbf\|>1-\delta$ where $ f= \sum_{g \in G} f_{g}$.
\end{enumerate} 
Using \eqref{outer_it1} and \eqref{outer_it2}  we have ($1$ denotes the neutral element of $G$):
\begin{align*}
\Big\|fb^{\frac{1}{2}}-\sum_{g\in G} \alpha_{g}(f_{1}b^{\frac{1}{2}})\Big\|&=
\Big\| \sum_{g \in G} f_{g}b^{\frac{1}{2}}-\sum_{g\in G} \alpha_{g}(f_{1})b^{\frac{1}{2}}\Big\| \\
&\leq \sum_{g \in G} \|f_{g}-\alpha_{g}(f_{1})\| < n\delta.
\end{align*}
Thus, using \eqref{outer_it3} at the third step we get
\[
n\delta>\|fb^{\frac{1}{2}}\|-\sum_{g\in G} \|\alpha_{g}(f_{1}b^{\frac{1}{2}})\|
= \|fbf\|^{\frac{1}{2}}-n\|f_{1}b^{\frac{1}{2}}\|>\sqrt{1-\delta}-n\|f_{1}b^{\frac{1}{2}}\|.
\]
Hence,
\begin{equation}\label{equout1}
\|f_{1}b^{\frac{1}{2}}\|>\dfrac{\sqrt{1-\delta}-n\delta}{n}>
\dfrac{\sqrt{1-\varepsilon}-n\varepsilon}{n}.
\end{equation}
By \eqref{outer_it1}, $\|f_{1}b-bf_{1}\|<\delta$ and so 
$\|f_{1}^{\frac{1}{2}}b-bf_{1}^{\frac{1}{2}}\|<\varepsilon$. Thus,
\begin{equation}\label{equout2}
\|f_{1}b-f_{1}^{\frac{1}{2}}bf_{1}^{\frac{1}{2}}\|\leq 
\|f_{1}^{\frac{1}{2}}b-bf_{1}^{\frac{1}{2}}\|<\varepsilon.
\end{equation}
Similarly, since 
$\|f_{g}b-bf_{g}\|<\delta$ we have
\begin{equation}\label{equout3}
\|f_{g_{0}}b-f_{g_{0}}^{\frac{1}{2}}bf_{g_{0}}^{\frac{1}{2}}\|<\varepsilon.
\end{equation}
Note that $f_{g_{0}}^{\frac{1}{2}}bf_{g_{0}}^{\frac{1}{2}}\perp
f_{1}^{\frac{1}{2}}bf_{1}^{\frac{1}{2}}$ and thus by (\ref{equout1}) we have
\begin{equation}\label{equout5}
\|f_{g_{0}}^{\frac{1}{2}}bf_{g_{0}}^{\frac{1}{2}}-f_{1}^{\frac{1}{2}}bf_{1}^{\frac{1}{2}}\|
\geq \|f_{1}^{\frac{1}{2}}bf_{1}^{\frac{1}{2}}\|=
\|f_{1}^{\frac{1}{2}}b^{\frac{1}{2}}\|^{2}\geq \|f_{1}b^{\frac{1}{2}}\|^{2}>
\left(\dfrac{\sqrt{1-\varepsilon}-n\varepsilon}{n}\right)^{2}.
\end{equation}
Moreover, using \eqref{outer_it1} we have
\begin{align}\label{equout4}
\|uf_{1}bu^{*}-bf_{1}\|&=\|uf_{1}bu^{*}-\alpha_{g_{0}}(b)f_{1}\|  \\ 
&=\|uf_{1}bu^{*}-ubu^{*}f_{1}\| \notag\\
&\leq \|f_{1}bu^{*}-bu^{*}f_{1}\|<\delta  \notag.
\end{align}
Finally, using \eqref{outer_it2}, (\ref{equout2}), (\ref{equout3}), (\ref{equout5}),   
and (\ref{equout4}) we obtain
\begin{align*}
\|\alpha_{g_{0}}(f_{1}b)-uf_{1}bu^{*}\|&\geq 
\|f_{g_{0}}^{\frac{1}{2}}bf_{g_{0}}^{\frac{1}{2}}-f_{1}^{\frac{1}{2}}bf_{1}^{\frac{1}{2}}\|-
\|f_{g_{0}}^{\frac{1}{2}}bf_{g_{0}}^{\frac{1}{2}}-f_{g_{0}}b\| - \|f_{g_{0}}b- \alpha_{g_{0}}(f_{1}b)\|\\
&\quad   -
\|uf_{1}bu^{*}-bf_{1}\| 
 - \|bf_{1}-f_{1}^{\frac{1}{2}}bf_{1}^{\frac{1}{2}}\|\\
&>\left(\dfrac{\sqrt{1-\varepsilon}-n\varepsilon}{n}\right)^{2}-\varepsilon-2\delta-\varepsilon\\
&>\left(\dfrac{\sqrt{1-\varepsilon}-n\varepsilon}{n}\right)^{2}-4\varepsilon>0,
\end{align*}
which is a contradiction. This shows that $\alpha$ is pointwise outer.
\end{proof}

\begin{cor}\label{simpleweak}
Let $\alpha$ be an action of a finite group $G$ on a simple C*-algebra $A$.  If $\alpha$ has the weak tracial Rokhlin property, then $A \rtimes _{\alpha} G$  is simple,
and hence the fixed point 
algebra $A^{\alpha}$ is isomorphic to a full corner of $A \rtimes _{\alpha} G$.
\end{cor}

\begin{proof}
It follows from~\cite[Theorem 3.1]{Kishimoto} and Proposition~\ref{outer}
that $A \rtimes _{\alpha} G$  is simple.
By  \cite{Rosenberg}, there exists a projection $p$ in the multiplier algebra 
of  $A \rtimes_{\alpha} G$ such that 
$A^{\alpha} \cong p(A\rtimes _{\alpha} G)p$. Since $A \rtimes _{\alpha} G$ 
is simple, 
$p(A\rtimes _{\alpha} G)p$ is a full corner. 
 \end{proof}

The following lemma shows that if
the property stated in Definition~\ref{w.t.R.p} holds for some
$y\in A_{+}$ (and every $x,F,\varepsilon$ there), then it also holds for any $z\in A_{+}$ 
which is ``smaller" than $y$
(that is, for any positive $z$ in $\overline{yAy}$).
(Note that $\overline{Ay}\cap A_{+}= \overline{yA}\cap A_{+}=\overline{yAy}\cap A_{+} $.)

\begin{lma}\label{lemdefwtrp}
Let $\alpha \colon G \rightarrow \mathrm{Aut}(A)$ be an action of a 
finite group $G$ on a simple C*-algebra $A$. 
Let $x \in A_{+}$ with $\|x\|=1$.
Suppose that a positive element $y \in A$
has the following property:
for every $\varepsilon > 0$ and  every finite subset $F \subseteq A$ 
there exists a family of orthogonal positive contractions
$(f_{g})_{g \in G}$ in $A$ such that, with $f=\sum_{g \in G} f_{g}$,
  the following hold:
\begin{enumerate}

\item\label{lemdefwtrp_it1}
$\|f_{g}a-af_{g}\| < \varepsilon$ for all $a\in F$ and all $g\in G$;
\item\label{lemdefwtrp_it2}
$\|\alpha_{g}(f_{h})-f_{gh}\| < \varepsilon$ for all $g,h\in G$;
\item\label{lemdefwtrp_it3}
$(y^{2}-yfy - \varepsilon)_{+} \precsim_{A} x$;
\item\label{lemdefwtrp_it4}
$\|fxf\| > 1-\varepsilon$.
 \setcounter{TmpEnumi}{\value{enumi}}
\end{enumerate}
Then every positive element $z\in\overline{Ay}$ also has the same property.
Moreover, the statement holds if we replace
``orthogonal positive contractions" with ``orthogonal projections."
\end{lma}

\begin{proof}
The idea of the proof is similar to the argument
given in the proof of \cite[Lemma~3.5]{AGJP17}.
Let $z \in \overline{Ay}$ be a positive element, and let a finite subset 
$F \subseteq A$,  $\varepsilon >0$, and an element $x \in A_{+}$ with $\|x\|=1$ be given.
Let $ \delta$ be such that $0<\delta < \min\left\{1, \frac{\varepsilon}{4(2\|z\|+1)}  \right\}$.
 Since $z \in \overline{Ay}$, 
there exists a nonzero element $w \in A$ such that $\|z-wy\| < \delta$.  
Choose $\eta>0$ such that $  \eta < \min \{\varepsilon, \frac{\varepsilon}{2\|w\|^{2}} \}$. 
 By assumption, there exists a family of  orthogonal positive contractions
  $(f_{g})_{g \in G}$ in $A$ such that,
  with $f=\sum_{g \in G} f_{g}$,
  the following hold:
\begin{enumerate}
 \setcounter{enumi}{\value{TmpEnumi}}
\item\label{lemdefwtrp_it5}
$\|f_{g}a-af_{g}\| < \eta$ for all $a\in F$ and all $g\in G$;
\item\label{lemdefwtrp_it6}
$\|\alpha_{g}(f_{h})-f_{gh}\| < \eta$ for all $g,h\in G$;
\item\label{lemdefwtrp_it7}
$(y^{2}-yfy - \eta)_{+} \precsim_{A} x$;
\item\label{lemdefwtrp_it8}
$\|fxf\| > 1-\eta$.
\end{enumerate}
 
 Since  $\eta < \varepsilon$, \eqref{lemdefwtrp_it5}, \eqref{lemdefwtrp_it6}, 
 and \eqref{lemdefwtrp_it8}  also hold for $\varepsilon$ in place of $\eta$. 
It remains to show that $(z^{2}-zfz-\varepsilon)_{+} \precsim_{A} x$. To see this, first  by 
Lemma~\ref{lemmkey} at the first step and by \eqref{lemdefwtrp_it7}
at the last step, we have
\begin{align*}
\left(wy^{2}w^{*}-wyfyw^{*}- {\varepsilon}/{2}\right)_{+} 
&\precsim_{A} w\left(y^{2}-yfy- \frac{\varepsilon}{2\|w\|^{2}}\right)_{+}w^{*}\\
&\precsim_{A} \left(y^{2}-yfy- \frac{\varepsilon}{2\|w\|^{2}}\right)_{+}\\
&\precsim_{A} (y^{2}-yfy-\eta)_{+} \precsim_{A} x.
\end{align*}
On the other hand, we have
\begin{align*}
\big\|z^{2}-zfz -&\left(wy^{2}w^{*}-wyfyw^{*}- \tfrac{\varepsilon}{2}\right)_{+}\big\|\\ 
&\leq \|z^{2}-zfz -(wy^{2}w^{*}-wyfyw^{*})\|+ \tfrac{\varepsilon}{2}\\
&\leq \|z^{2}-wyz\|+\|wyz-wy^{2}w^{*}\|+\|zfz-wyfz\|\\
&\quad +\|wyfz-wyfyw^{*}\|+\tfrac{\varepsilon}{2} \\
&\leq \delta (2\|z\|+2\|wy\|)+ \tfrac{\varepsilon}{2}\\ 
&\leq \delta(4\|z\|+2\delta)+ \tfrac{\varepsilon}{2} \leq \delta(4 \|z\|+2)+ \tfrac{\varepsilon}{2} <  \varepsilon.
\end{align*}
Therefore, by Lemma~\ref{lemkr}, 
 $(z^{2}-zfz-\varepsilon)_{+} \precsim_{A} \left(wy^{2}w^{*}-wyfyw^{*}- \frac{\varepsilon}{2}\right)_{+} 
 \precsim_{A} x$, 
 as desired. This finishes the proof.
\end{proof}

\begin{rem}\label{rmkdefwtrp}
Let $\alpha \colon G \rightarrow \mathrm{Aut}(A)$ be an action of a finite group $G$ on a simple 
 C*-algebra $A$. 
\begin{enumerate}
\item\label{rmkdefwtrp_it2}
If $A$ is $\sigma$-unital  then $\alpha$ has
the (weak) tracial Rokhlin property if \emph{some} strictly positive element $y$ in $A$ has the property stated in 
Definition~\ref{w.t.R.p}. This follows from Lemma~\ref{lemdefwtrp} 
and  $A=\overline{yAy}=\overline{Ay}$.
 \item\label{rmkdefwtrp_it3}
 In Definition~\ref{w.t.R.p}, it is enough to take $y$ in a norm
dense subset of $A_{+}$. Moreover, if $(e_{i})_{i\in I}$ is a approximate identity
for $A$, it is enough to take $y$ from the set $\{e_{i} \colon i\in I \}$. This follows
from Lemma~\ref{lemdefwtrp} and the fact that the set
$\left\{y \in A_{+} \colon y\in \overline{Ae_{i}} \ \text{for some}\ i\in I \right\}$ 
is norm dense in $A_{+}$. 
\item\label{rmkdefwtrp_it4}
 In  Definition~\ref{w.t.R.p}, if moreover $A$ is purely infinite then 
 Condition~\eqref{w.t.R.p_it3} is automatic.
 Also, if $A$ is finite then Condition~\eqref{w.t.R.p_it4}  is redundant
 (this is proved in \cite{AGJP17}, however, we do not need it here).
 
  \item\label{rmkdefwtrp_it5}
 If moreover $A$ is unital, then
$\alpha$ has the  tracial Rokhlin property (in the sense of Definition~\ref{w.t.R.p})
if and only if the conditions of Definition~\ref{w.t.R.p} hold only for $y=1$ (and 
every $\varepsilon,F,x$ as in that definition). This implies that 
our definition of the  tracial Rokhlin property
 and  \cite[Definition~1.2]{Ph11} are equivalent in the 
unital case.
\end{enumerate}
\end{rem}

The Rokhlin property for finite group actions on \emph{arbitrary}
 C*-algebras 
 has been introduced
 in \cite{Sa15}.

\begin{prop}\label{propsan}
Let $A$ be a simple C*-algebra and let $\alpha \colon G\to \mathrm{Aut}(A)$ be an action of a 
finite group $G$ on $A$. If $\alpha$ has the Rokhlin property in the sense of
Definition~3.2 of \cite{Sa15},  then it
has the weak tracial Rokhlin property.
\end{prop}

\begin{proof}
Let $\alpha$ have the Rokhlin property in the sense of Definition~3.2 of \cite{Sa15}.
Let $x,y\in A_{+}$ with $\|x\|=1$, let $F\subseteq A$ be a finite subset, and 
let $\varepsilon>0$. We may assume that
$y\neq 0$ and $x,y\in F$. Also, by Lemma~\ref{lemdefwtrp}, we may 
further assume that $\|y\| \leq 1$.
Since $\alpha$ has the Rokhlin property,
there exists  a family of orthogonal positive contractions $(f_{g})_{g\in G}$ in $A$ 
 such that,
 with $f=\sum_{g\in G}f_{g}$, we have:
\begin{enumerate}
\item\label{propsan_it1}
$\|\alpha_{g}(f_{h})-f_{gh}\|<\frac{\varepsilon}{2}$ for all $g,h\in G$;
\item\label{propsan_it2}
$\|f_{g}a- af_{g}\|<\frac{\varepsilon}{2}$ for all $g\in G$ and all $a\in F$;
\item\label{propsan_it3}
$\| fa-a\|<\frac{\varepsilon}{2}$ for all $a\in F$.
\end{enumerate}
Then, Conditions~\eqref{w.t.R.p_it1} and \eqref{w.t.R.p_it2} in Definition~\ref{w.t.R.p} 
are satisfied (by \eqref{propsan_it1} and \eqref{propsan_it2} above).
 Since $y\in F$, by \eqref{propsan_it3} we have
\[
\|y^{2}-yfy\|\leq \|y-yf\|<{\varepsilon}/{2}<\varepsilon.
\]
Thus $(y^{2}-yfy-\varepsilon)_{+}=0\precsim_{A} x$. Hence, Condition~\eqref{w.t.R.p_it3} 
in  Definition~\ref{w.t.R.p} is also satisfied. 
To prove Condition~\eqref{w.t.R.p_it4},  by \eqref{propsan_it3} and that $x\in F$ we have
\[
\|fxf-x\|\leq \|fxf-xf\|+\|xf-x\| \leq \|fx-x\|+\|xf-x\|<\varepsilon.
\]
Thus $\|fxf\|> \|x\|-\varepsilon=1-\varepsilon$,
and so Condition~\eqref{w.t.R.p_it4} 
in  Definition~\ref{w.t.R.p} holds. 
Therefore, $\alpha$ has the weak tracial Rokhlin property.
\end{proof}

There are several weaker versions of the tracial Rokhlin property for actions 
on simple unital C*-algebras. In the sequel, we   compare them 
with our definition of the weak tracial Rokhlin property given in
Definition~\ref{w.t.R.p}. First we recall the following definition for the convenience of the reader.
(See \cite[Definition~2.2]{GHS} for an
 equivalent definition.)

\begin{df}[see \cite{GH20}]\label{unitalgar}
Let $\alpha \colon G \rightarrow \mathrm{Aut}(A)$ be an action of 
a finite group $G$ on a simple unital  C*-algebra 
$A$. Then $\alpha$ has the \emph{weak tracial Rokhlin property} if 
for every $\varepsilon > 0$,  
every finite subset $F \subseteq A$, and every positive element $x \in A$ 
with $\|x\|=1$, 
there exists a family of orthogonal positive contractions $(f_{g})_{g \in G}$ in 
$A$ such that, with $f=\sum_{g \in G} f_{g}$,
  the following hold:
\begin{enumerate}

\item\label{unitalgar_it1}
$\|f_{g}a-af_{g}\| < \varepsilon$ for all $a\in F$ and all $g\in G$;
\item\label{unitalgar_it2}
$\|\alpha_{g}(f_{h})-f_{gh}\| < \varepsilon$ for all $g,h\in G$;
\item\label{unitalgar_it3}  
$1-f \precsim_{A} x$;
\item\label{unitalgar_it4} 
$\|fxf\| > 1-\varepsilon$.

\end{enumerate} 
\end{df}

\begin{prop}\label{propequiw}
Let $\alpha \colon G \rightarrow \mathrm{Aut}(A)$ be an action of a finite group $G$ on a simple unital
C*-algebra $A$. The following statements are equivalent:
\begin{itemize}
\item[(a)]
$\alpha$ has the weak tracial Rokhlin property in the sense of Definition~\ref{unitalgar};
\item[(b)]
$\alpha$ has the weak tracial Rokhlin property in the sense of Definition~\ref{w.t.R.p}.
\end{itemize}

\end{prop}

\begin{proof}
The implication 
(a)$\Rightarrow$(b)
follows from  Remark~\ref{rmkdefwtrp}\eqref{rmkdefwtrp_it2}
(which implies that 
it is enough to take $y=1$ in 
Definition~\ref{w.t.R.p}) and the fact that 
$(1-f-\varepsilon)_{+} \precsim_{A} 1-f$.

To show  (b)$\Rightarrow$(a),
one may take small cut-downs of $f_{g}$'s in
Definition~\ref{w.t.R.p} to get the desired Cuntz-subequivalence.
We give an alternative proof using the functional calculus
for order zero maps.
 Let $F$, $x$, and
$\varepsilon$ be as in Definition~\ref{unitalgar}. 
We will find orthogonal positive contractions $(f_{g})_{g \in G}$ in $A$ 
satisfying \eqref{unitalgar_it1}--\eqref{unitalgar_it4} of 
Definition~\ref{unitalgar}.  
 We may assume that $F$ is contained in the closed unit ball of $A$.
 Set $n=\card(G)$.
Choose $\delta$ with $ 0< \delta < \tfrac{\varepsilon}{2n+1}$.
Applying Definition~\ref{w.t.R.p} with $\delta$ in place of $\varepsilon$, with $y=1$, 
and with $x, F$ as given, there are orthogonal positive contractions $(e_{g})_{g \in G}$ in $A$ 
such that, with $e= \sum _{g \in G} e_{g}$, the following hold:
\begin{enumerate}
\item\label{propequiw_pf_it1}
$\|e_{g}a-ae_{g}\| < \delta$ for all $a\in F$ and all $g\in G$;
\item\label{propequiw_pf_it2} 
$\|\alpha_{g}(e_{h})-e_{gh}\| < \delta$ for all $g,h\in G$;
\item\label{propequiw_pf_it3} 
$(1-e- \delta)_{+} \precsim_{A} x$;
\item\label{propequiw_pf_it4}  
$\|exe\| > 1-\delta$.
\setcounter{TmpEnumi}{\value{enumi}}
\end{enumerate}
We define a c.p.c.~order zero map $\phi\colon C(G) \rightarrow A$ 
by $\phi(\xi)= \sum _{g\in G} \xi(g) e_{g}$. 
Then $\phi(1)=e$. Let $\eta\colon [0,1] \rightarrow [0,1]$ be the continuous function defined by 
\[
\eta(\lambda)=
\begin{cases}
 (1-\delta)^{-1} \lambda & 0 \leq \lambda \leq 1- \delta,\\
 1& 1-\delta< \lambda \leq 1.
 \end{cases}
\]
The graph of $\eta$ is the following:

\begin{center}
\begin{tikzpicture}[scale=2]
\draw [->] (-0.2,0) --(1.4,0);
\draw [->] (0,-0.2) --(0,1.4);

\draw (1,-0.05) --(1,0.05);
\draw (0.7,-0.05) --(0.7,0.05);
\node at (1,-0.2) {\tiny{$1$}};
\node at (0.7,-0.2) {\tiny{$1-\delta$}};
\draw (-0.05,1) --(0.05,1);
\node at (-0.15,1) {\tiny{$1$}};

\draw[thick] (0,0) --(0.7,1);
\draw[thick] (0.7,1) --(1,1);
\draw[dashed] (1,0) --(1,1);
\node at (0.3,0.7) {{$\eta$}};
\end{tikzpicture}
\end{center}

Using the functional calculus for c.p.c.~order zero maps (\cite[Corollary~4.2]{WZ08}), 
we define $\psi= \eta(\phi)$. Thus $\psi\colon C(G) \rightarrow A$ is a c.p.c.~order zero map. 
Similar to the argument given in the proof of \cite[Lemma~2.8]{ABP17}, 
we see that $\|\psi(z)- \phi(z)\| \leq \delta \|z\|$ for all $z \in C(G)$ 
with $\|z\|=1$, and that 
\[ 
1-\psi(1) = \tfrac{1}{1-\delta} \left(1- \phi(1)-\delta \right)_{+} \sim_{A} 
\left(1 -\phi(1)- \delta \right)_{+} \precsim_{A} x.
\]
For any $g\in G$, set $f_{g} = \psi(\chi_{\{g\}})$.  
 Thus, $(f_{g})_{g \in G}$ is a family of orthogonal positive contractions in 
 $A$ and we have 
 \begin{enumerate}
\setcounter{enumi}{\value{TmpEnumi}}
\item\label{propequiw_pf_it5} 
 $\|f_{g}-e_{g}\| \leq \delta$ for all $g \in G$.
 \end{enumerate}
 Moreover, with $ f= \sum _{g \in G} f_{g}$, we have $1-f=1- \psi(1) \precsim_{A} x$, which is 
 \eqref{unitalgar_it3} in Definition~\ref{unitalgar}.  
Using \eqref{propequiw_pf_it5}, it is easy to see that  Conditions~\eqref{unitalgar_it1}, \eqref{unitalgar_it2},
and \eqref{unitalgar_it4} in Definition~\ref{unitalgar}  follow
from \eqref{propequiw_pf_it1}, \eqref{propequiw_pf_it2}, and \eqref{propequiw_pf_it4} above,
respectively. 
\end{proof}

In the following lemma,
we give a (seemingly) stronger equivalent definition of the 
(weak) tracial Rokhlin property for finite group 
actions on simple C*-algebras. This lemma says that we can take
two different unknowns $x,z$ in Conditions~\eqref{w.t.R.p_it3} and
\eqref{w.t.R.p_it4} of Definition~\ref{w.t.R.p} instead of  $x$.
The idea of the proof of this lemma will  be used also in a number of places later.

\begin{lma}\label{lemtrp2}
Let $\alpha \colon G \rightarrow \mathrm{Aut}(A)$ be an action of a finite group $G$ on a simple 
C*-algebra $A$.
Then  $\alpha$ has the weak tracial Rokhlin property
(respectively, tracial Rokhlin property) if and only if 
the following holds. For
every  $\varepsilon > 0$,  every finite subset 
$F \subseteq A$, and all positive elements $x,y,z \in A$ with $x\neq 0$ and $\|z\|=1$,
there exists a family of orthogonal  positive contractions 
(respectively, orthogonal projections)
$(f_{g})_{g \in G}$ in $A$  such that, with $f=\sum_{g \in G} f_{g}$,
 the following hold:
\begin{enumerate}
\item\label{lemtrp2_it1}
$\|f_{g}a-af_{g}\| < \varepsilon$ for all $a\in F$ and all $g\in G$;
\item\label{lemtrp2_it2} 
$\|\alpha_{g}(f_{h})-f_{gh}\| < \varepsilon$ for all $g,h\in G$;
\item\label{lemtrp2_it3}  
$(y^{2}-yfy - \varepsilon)_{+} \precsim_{A} x$;
\item\label{lemtrp2_it4} 
$\|fzf\| > 1-\varepsilon$.
\setcounter{TmpEnumi}{\value{enumi}}
\end{enumerate}
\end{lma}

\begin{proof}
We prove only the case of the weak tracial
Rokhlin property since the proof for the tracial Rokhlin property
is  similar.
The backward implication is obvious. 
For the forward implication, let $\alpha$
have the weak tracial Rokhlin property and let $\varepsilon, F, x,y,z$ be as in the statement.
We may assume that $F$ is contained in the closed unit ball of $A$. Let $n=\card(G)$.
Choose $\delta$  such that $0<\delta<1$ and  
\[
\left(\dfrac{\delta}{2-\delta}(1-\tfrac{\varepsilon}{2})\right)^{2}>1-\varepsilon.
\]
Put $z_{1}=(z^{{1}/{2}}-\delta)_{+}$.
Since $A$ is simple, \cite[Lemma~2.6]{Ph14} implies that there is a positive element
$d\in \overline{z_{1}Az_{1}}$ such that $d\precsim_{A} x$ and $\|d\|=1$.
Applying Definition~\ref{w.t.R.p} with $y$ and $F$ as given, with $\frac{\varepsilon}{2}$ in place of $\varepsilon$,
and with $d$ in place of $x$, there exist orthogonal positive contractions 
$(f_{g})_{g \in G}$ in $A$ such that, with $f=\sum_{g \in G} f_{g}$,
  the following hold:
\begin{enumerate}
\setcounter{enumi}{\value{TmpEnumi}}
\item\label{lemtrp2_pf_it5}
 $\|f_{g}a-af_{g}\| < \frac{\varepsilon}{2}$ for all $a\in F$ and all $g\in G$;
\item\label{lemtrp2_pf_it6}
 $\|\alpha_{g}(f_{h})-f_{gh}\| < \frac{\varepsilon}{2}$ for all $g,h\in G$;
\item\label{lemtrp2_pf_it7}
 $(y^{2}-yfy - \frac{\varepsilon}{2})_{+} \precsim_{A} d$;
\item\label{lemtrp2_pf_it8}
 $\|fdf\| > 1-\frac{\varepsilon}{2}$.
\end{enumerate}
Clearly, \eqref{lemtrp2_it1}, \eqref{lemtrp2_it2}, and \eqref{lemtrp2_it3} follow  from 
 \eqref{lemtrp2_pf_it5},  \eqref{lemtrp2_pf_it6}, and \eqref{lemtrp2_pf_it7}, respectively.
To see \eqref{lemtrp2_it4}, first note that have 
$d\in \overline{z_{1}Az_{1}}\subseteq \overline{Az_{1}}=\overline{A(z^{{1}/{2}}-\delta)_{+}}$.
Thus by Lemma~\ref{lembdd} there exists a sequence $(v_{n})_{n\in \mathbb{N}}$ in $A$
such that $\|v_{n}z^{{1}/{2}}-d\|\to 0$ and
$\|v_{n}\|\leq (\|d\|+\frac{1}{n})\delta^{-1}=(1+\frac{1}{n})\delta^{-1}$.
Then $\|fv_{n}z^{\frac{1}{2}}f-fdf\|\to 0$. Since $\|fdf\|>1-\frac{\varepsilon}{2}$ and 
$\delta<1$, there is $n\in \mathbb{N}$ such that 
$\|fv_{n}x^{\frac{1}{2}}f\|>1-\frac{\varepsilon}{2}$ and $\frac{1}{n}<1-\delta$. Hence,
\[
1-\tfrac{\varepsilon}{2}<\|fv_{n}z^{\frac{1}{2}}f\|\leq \|z^{\frac{1}{2}}f \|
\cdot  \|v_{n}\|
\leq \|z^{\frac{1}{2}}f \| (1+\tfrac{1}{n})\delta^{-1}\leq 
\|z^{\frac{1}{2}}f \| (2-\delta)\delta^{-1}.
\]
Thus, 
\[
\|fzf\|=\|z^{\frac{1}{2}}f\|^{2}>
\left(\dfrac{\delta}{2-\delta}(1-\tfrac{\varepsilon}{2})\right)^{2}>1-\varepsilon.
\]
This completes the proof.
\end{proof}

Phillips in \cite[Problem~3.2]{Ph.free} asked whether there is a reasonable formulation 
of the tracial Rokhlin property for finite group actions on simple unital  C*-algebras
in terms of the central sequence algebra. 
We give an answer to this question in the \emph{not necessarily unital} simple case 
in the following proposition. 
For a C*-algebra $A$, we write 
\[
A_{\infty}= {\ell^{\infty}(\mathbb{N}, A)} /{c_{0}(\mathbb{N}, A)}.
\]
We consider the elements of $A$  in $A_{\infty}$ as the equivalence classes of  constant sequences. 
We denote by $A_{\infty} \cap A'$ the relative commutant of $A$ in $A_{\infty}$. 
Also, $\pi_{\infty} \colon \ell^{\infty}(\mathbb{N}, A)\to A_{\infty}$
denotes the quotient map.
If $\alpha \colon G \rightarrow \mathrm{Aut} (A)$  is an action, then we denote by
 $\alpha _{\infty}$ the induced action of $G$ on $A_{\infty}$.
 
\begin{prop}\label{propcs}
Let  $\alpha \colon G \rightarrow \mathrm{Aut} (A)$ be 
an action of a finite group $G$ on a simple separable C*-algebra $A$. 
Then $\alpha$ has the weak tracial Rokhlin property 
(respectively, tracial Rokhlin property)
if and 
only if for every  $x, y,z \in A_{+}$ with $x\neq 0$, there exists a family of orthogonal  
positive contractions 
(respectively, orthogonal projections)
$(f_{g})_{g \in G}$ in $ A _{\infty} \cap A'$ such that, with $f = \sum_{g \in G} f_{g}$, 
the following hold:
\begin{enumerate}
\item\label{propcs_it1} 
$(\alpha_{\infty})_{g}(f_{h})=f_{gh}$ for all $g,h\in G$;

\item\label{propcs_it2}  
$y^{2}-yfy \precsim_{A_{\infty}} x$;

\item\label{propcs_it3}   
$\|fzf\|=\|z\|$.
\end{enumerate}
If moreover $A$ is unital, then Condition~\eqref{propcs_it2} can replaced by
$1-f\precsim_{A_{\infty}} x$.
\end{prop}

\begin{proof}
We prove only the case of the weak tracial
Rokhlin property since the proof for the tracial Rokhlin property
is essentially the same.

Assume that $\alpha$ has the weak tracial Rokhlin property. 
Let $x, y,z \in A_{+}$ with $x\neq 0$.
We may assume that $\|x\|=1$ and $z\neq 0$.
Let $\{a_{1}, a_{2}, \ldots\}$ be a norm dense countable subset of 
the closed unit ball of $A$.  For $n\in \mathbb{N}$, set $F_{n}=\{a_{1},\ldots,a_{n}\}$. 
Applying Lemma~\ref{lemtrp2} with $(x-\tfrac{1}{2})_{+}$ in place of $x$, with
$z/\|z\|$ in place of $z$, with $F_{n}$ in place of $F$, and with $\tfrac{1}{n}$ in place of $\varepsilon$, 
we obtain mutually orthogonal positive contractions $(f_{(g,n)})_{g \in G}$ in $A$ 
satisfying \eqref{lemtrp2_it1}--\eqref{lemtrp2_it4} in 
Lemma~\ref{lemtrp2}. Let $f_{g}\in A_{\infty}$ denote the equivalence class of 
$(f_{(g,n)})_{n \in \mathbb{N}}$.
Then  
 $(f_{g})_{g \in G}$ is a family of  orthogonal  
positive contractions in $A _{\infty} \cap A'$. It is easy to see  that 
Conditions~\eqref{propcs_it1} and \eqref{propcs_it3} in the statement hold. 
To see \eqref{propcs_it2},  put $h_{n}=\sum_{g \in G} f_{(g,n)}$. 
Then $f=\sum_{g \in G} f_{g}$ is the equivalence class of $(h_{n})_{n \in \mathbb{N}}$ in $A_{\infty}$ 
and $(y^{2}-yh_{n}y-\tfrac{1}{n})_{+} \precsim_{A} (x-\tfrac{1}{2})_{+}$.  By Lemma~\ref{lembdd1}, 
there is $v_{n} \in A$ such that $\|(y^{2}-yh_{n}y-\tfrac{1}{n})_{+}- v_{n}xv_{n}^{*}\|< \tfrac{1}{n}$ 
and $\|v_{n}\| \leq 2\|y\|$. Hence $\|y^{2}-yh_{n}y-v_{n}xv_{n}^{*}\|< \tfrac{2}{n}$. 
Let $v$ be the equivalence class of $(v_{n})_{n \in \mathbb{N}}$ in $A_{\infty}$. 
(Note that $(v_{n})_{n \in \mathbb{N}}$ is a bounded sequence.)
Then 
$y^{2}-yfy=vxv^{*} \precsim_{A_{\infty}} x$.

Using Lemma~\ref{lemkr} and \cite[Lemma~2.5.12]{Lin.book}, the other implication follows.

If moreover $A$ is unital and we replace 
Condition~\eqref{propcs_it2}   by
$1-f\precsim_{A_{\infty}} x$, then the forward implication follows by taking $y=1$ in the 
preceding argument. The backward implication follows from the previous case
and the fact that $y^{2}-yfy =y(1-f)y \precsim_{A_{\infty}} 1-f$.
\end{proof}

The following theorem says
that for finite group actions on simple C*-algebras with
tracial rank zero, the weak tracial Rokhlin property coincides
with the tracial Rokhlin property.
Similar results were proved  in
\cite[Theorem~1.9]{Ph15} and \cite[Theorem~2.7]{Wang18} for the 
unital case. In the proof of these results
 the tracial state space was used as an ingredient. However, in our nonunital setting, we use techniques from
Cuntz subequivalence in the proof of the following theorem instead of using  traces.

We use the following fact
(which is easy to prove) in the proof of the following theorem.
If $A$ is a C*-algebra with real rank zero and
$B$ is a finite dimensional C*-subalgebra of $A$,
then both $A_{\infty}$ and $A\cap B'$
have real rank zero
  (see \cite[Theorem~4.10(iv)]{DE02}). 

\begin{thm}\label{wtr=tr}
Let $A$ be a simple  C*-algebra with tracial rank zero, 
and let $\alpha \colon G \rightarrow \Aut(A)$ be an action of a finite group $G$ on $A$. 
If $\alpha$ has the weak tracial Rokhlin property 
then it has the  tracial Rokhlin property.
\end{thm}

\begin{proof}
Suppose that $\alpha$ has the weak tracial Rokhlin property. 
We have to show that
for  every finite subset $F \subseteq A$, every $\varepsilon > 0$, and all positive 
elements $x,y \in A$ with $\|x\|=1$,  there exists a family of orthogonal projections 
$(p_{g})_{g \in G}$ in $A$ such that, with $p=\sum_{g \in G} p_{g}$,
the following hold:
\begin{enumerate} 
\item\label{it1.wtrp=trp}
$\|p_{g}a-ap_{g}\| < \varepsilon$ for all $a\in F$ and all $g\in G$;
\item\label{it2.wtrp=trp}
 $\|\alpha_{g}(p_{h})-p_{gh}\| < \varepsilon$ for all $g,h\in G$;
\item\label{it3.wtrp=trp}
  $(y^{2}-ypy - \varepsilon)_{+} \precsim_{A} x$;
\item\label{it4.wtrp=trp}
 $\|pxp\| > 1-\varepsilon$.
 \setcounter{TmpEnumi}{\value{enumi}}
 \end{enumerate} 
 
 Set $n=\card(G)$.
Choose $\delta >0$ with $\delta < \min( \varepsilon/8, \varepsilon/2(n+1))$. 
Without loss of generality, we can assume that $x, y \in F$, that $F$ is 
contained in the closed unit ball of $A$, and that $\alpha_{g}(F)=F$, for any $g \in G$. 
We may further assume that $y \in A^{\alpha}$ (by Remark~\ref{rmkdefwtrp}\eqref{rmkdefwtrp_it3}
since $A$ has an 
approximate identity contained in $A^{\alpha}$). 
We claim that there exists $z \in A_{+}\setminus \{0\}$ such that
\begin{enumerate}
\setcounter{enumi}{\value{TmpEnumi}}

\item \label{it5.wtrp=trp}
 $z\oplus \bigoplus_{g \in G} \alpha_{g}(z) \precsim_{A} x$.

\setcounter{TmpEnumi}{\value{enumi}}
\end{enumerate}
In fact, by~\cite[Lemma 2.1]{Ph14}, there is $ z_{1} \in A_{+} \setminus \{0\}$ such that 
$z_{1} \otimes 1_{n+1} \precsim_{A} x$. 
Then by \cite[Lemma 2.4]{Ph14}, there is $z \in A_{+} \setminus \{0\}$ 
such that $z \precsim_{A} \alpha_{g}(z_{1})$ for all $g \in G$. 
We may assume that $\|z\|=1$. Hence, 
$z \oplus \bigoplus_{g \in G} \alpha_{g}(z) \precsim_{A} z_{1} \otimes 1_{n+1} \precsim_{A} x$. 
This proves \eqref{it5.wtrp=trp}.

Since $A$ has tracial rank zero, we can apply Theorem~\ref{main.trk} to find a finite dimensional 
C*-subalgebra $B$ of $A$ such that, with $q=1_{B}$, the following hold:
\begin{enumerate} \setcounter{enumi}{\value{TmpEnumi}}
\item \label{it6.wtrp=trp}
$\|qa -aq\| < \delta$ for all $a \in F$;
\item \label{it7.wtrp=trp}
$qAq \subseteq _{\delta} B$;
\item \label{it8.wtrp=trp}
$(y^{2} - yqy-\delta)_{+} \precsim_{A} z$;
\item \label{it9.wtrp=trp}
$\|qxq\| > 1-\delta$.
 \setcounter{TmpEnumi}{\value{enumi}}
 \end{enumerate}

Set $E=F\cup B$, and let $B_{0}$ be a finite subset
of $B$ with $\mathrm{span}(B_{0})=B$.
Since $\alpha$ has the weak tracial Rokhlin property,   
arguing as in the proof of Proposition~\ref{propcs} 
 with $z$ in place of $x$, with $y$ as given, and with
$ qxq$ in place of $z$ (using
$F\cup B_{0}$ in place of $F$
when applying Lemma~\ref{lemtrp2} in the proof of Proposition~\ref{propcs}),
there are 
orthogonal positive contractions $(f_{g})_{g \in G}$ in 
$A_{\infty} \cap (F\cup B_{0})'=A_{\infty} \cap E'$
 such that, 
with $f= \sum_{g \in G} f_{g}$, we have 
\begin{enumerate}
\setcounter{enumi}{\value{TmpEnumi}}
\item \label{it10.wtrp=trp}
$(\alpha_{\infty})_{g}(f_{h})=f_{gh}$ for all $g, h \in G$;
\item \label{it11.wtrp=trp}
$y^{2}-yfy \precsim_{A_{\infty}} z$;
\item \label{it12.wtrp=trp} 
$\|fqxqf\|= \|qxq\| >1-\delta$ (by \eqref{it9.wtrp=trp}).

\setcounter{TmpEnumi}{\value{enumi}}
\end{enumerate}
Note that $\|fqxqf\|= \max \{\|f_{g}qxqf_{g}\| \colon g \in G\|\}$. Hence, by \eqref{it12.wtrp=trp}, 
there exists $g_{0} \in G$ such that 
\begin{equation}\label{inq1.wtrp=trp}
\|f_{g_{0}}qxqf_{g_{0}}\| > 1-\delta.
\end{equation}
 Put $D = A_{\infty} \cap B'$. Note that $A_{\infty} \cap E' \subseteq D$ 
 and that $D$ has real rank zero
 by the remark preceding this theorem. 
 In particular, the hereditary subalgebra $\overline{qf_{g_{0}}D f_{g_{0}}q}$ 
 of $D$ has real rank zero.  Thus there is a projection 
 $r_{1} \in \overline{qf_{g_{0}}Df_{g_{0}}q}$ such that  
\[
 \|r_{1}qf_{g_{0}}-qf_{g_{0}}\|<\delta\ \  \text{and} \ \ \|r_{1}qf_{g_{0}}r_{1}-qf_{g_{0}}\|<\delta.
\]
 Since $r_{1} \leq q$, we get 
 \begin{equation} \label{inq2.wtrp=trp}
 \|r_{1}f_{g_{0}}-qf_{g_{0}}\|<\delta\ \  \text{and} \ \   \|r_{1}f_{g_{0}}r_{1}-qf_{g_{0}}\|<\delta. 
 \end{equation}
 Put $r_{g}=\alpha_{g}(r_{1})$ for every $g \in G\setminus\{1\}$, 
 and $r =\sum_{g \in G} r_{g}$. Thus $(r_{g})_{g \in G}$ 
 is a family of projections in $A_{\infty}$. Since 
 $r_{1} \in \overline{f_{g_{0}}Df_{g_{0}}} \subseteq \overline{f_{g_{0}}A_{\infty}f_{g_{0}}}$, we have $r_{g} \in 
 \overline{f_{gg_{0}}A_{\infty}f_{gg_{0}}}$, and so $r_{g}r_{h}=0$ 
 when $g \neq h$; that is, $(r_{g})_{g \in G}$ is a family of orthogonal 
 projections in $A_{\infty}$. We show that
 
 \begin{enumerate}
\setcounter{enumi}{\value{TmpEnumi}}
\item \label{it13.wtrp=trp}
$\|r_{g}a-ar_{g}\| <\varepsilon/2$ for all $a \in F$ and all $g \in G$,
\item \label{it14.wtrp=trp}
$(\alpha_{\infty})_{g}(r_{h})=r_{gh}$ for all $g, h \in G$,
\item \label{it15.wtrp=trp}
$(y^{2}-yry-\varepsilon/2)_{+} \precsim_{A_{\infty}} x$, and
\item \label{it16.wtrp=trp}
$\|rxr\| >1-\varepsilon/2$.

\setcounter{TmpEnumi}{\value{enumi}}
\end{enumerate}
 Observe that~\eqref{it14.wtrp=trp} follows from the definition of $r_{g}$.  
 For \eqref{it13.wtrp=trp}, let $a \in F$. By \eqref{it7.wtrp=trp}, 
 there is $b \in B$ such that $\|qaq-b\| <\delta$. Using this at the third step, 
 and \eqref{it6.wtrp=trp} at the fourth step, we get
 \begin{align*} 
 \|r_{1}a-ar_{1}\|&=\|qr_{1}qa-aqr_{1}q\| \\
 &\leq \|qr_{1}qa-qr_{1}qaq\|+\|qr_{1}qaq-qaqr_{1}q\| +\|qaqr_{1}q-aqr_{1}q\| \\
 & \leq \|qr_{1}\|\cdot\|qa-qaq\|+2\delta+\|qaq-aq\|\cdot\|r_{1}q\| \\
 & \leq 4\delta <\varepsilon/2.
 \end{align*}
 Thus $\|r_{g}a-ar_{g}\| < \varepsilon/2$, for all $a \in F$ and all $g\in G$. 
 Now, using~\eqref{it14.wtrp=trp} and that $\alpha_{g}(F)=F$, for all $g \in G$, 
 we get~\eqref{it13.wtrp=trp}. To see~\eqref{it15.wtrp=trp}, first by
 \eqref{inq2.wtrp=trp} at the fifth step (also recall that $y \in A^{\alpha}$ 
 and $f_{g} \in A_{\infty}\cap E'$), we get 
 \begin{align}\label{inq3.wtrp=trp}
 \notag \Big\|(y^{2}-yrfry)&- \Big[(y^{2}-yfy)+
  \sum_{g \in G}f_{gg_{0}}^{1/2}\alpha_{g}((y^{2}-yqy-\delta)_{+})f_{gg_{0}}^{1/2}\Big]\Big\| \\\notag
 & \leq \Big\|yfy-yrfry-\sum_{g \in G}f_{gg_{0}}^{1/2}\alpha_{g}(y^{2}-yqy)f_{gg_{0}}^{1/2}\Big\|
 +\delta  \\ \notag
  &\leq \Big\|yfy-\sum_{g \in G}f_{gg_{0}}^{1/2} y^{2}f_{gg_{0}}^{1/2}\Big\|
  +\Big\|yrfry-\sum_{g \in G}yf_{gg_{0}}^{1/2}\alpha_{g}(q)f_{gg_{0}}^{1/2}y\Big\|+\delta \\
  &=\Big\|\sum_{g \in G}yr_{g}f_{gg_{0}}r_{g}y -\sum_{g \in G}y \alpha_{g}(q)f_{gg_{0}}y\Big\|
  +\delta  \\\notag
 &\leq \|y^{2}\|\cdot\Big \|\sum_{g \in G} \alpha_{g}(r_{1}f_{g_{0}}r_{1})-\alpha_{g}(qf_{g_{0}})\Big\|
 +\delta  \\
&\leq n  \delta + \delta=(n+1)\delta <\varepsilon/2  \notag.
 \end{align}
 On the other hand, $yrfry \leq yry$ and so $y^{2}-yry \leq y^{2}-yrfry$. 
 Then \cite[Lemma~1.7]{Ph14} implies that
 $(y^{2}-yry-\varepsilon/2)_{+} \precsim_{A_{\infty}} (y^{2}-yrfry-\varepsilon/2)_{+}$. 
 By this at the first step, by \eqref{inq3.wtrp=trp} at the second step, by
 \eqref{it8.wtrp=trp} and \eqref{it11.wtrp=trp} at the third step, and by
 \eqref{it5.wtrp=trp} at the fifth step, we get
 \begin{align*}
 (y^{2}-yry-\varepsilon/2)_{+}&\precsim_{A_{\infty}} (y^{2}-yrfry-\varepsilon/2)_{+}\\
 & \precsim_{A_{\infty}} (y^{2}-yfy)+\sum_{g \in G}f_{gg_{0}}^{1/2}\alpha_{g}((y^{2}-yqy-\delta)_{+})f_{gg_{0}}^{1/2} \\
 &\precsim_{A_{\infty}} z\oplus \sum_{g \in G}f_{gg_{0}}^{1/2} \alpha_{g}(z) f_{gg_{0}}^{1/2} \\
 &\precsim_{A_{\infty}} z\oplus \bigoplus_{g \in G}  \alpha_{g}(z)\\
 & \precsim_{A_{\infty}} x,
  \end{align*} 
 which is \eqref{it15.wtrp=trp}.
 To prove \eqref{it16.wtrp=trp}, by \eqref{inq2.wtrp=trp} at the sixth step and 
 by \eqref{inq1.wtrp=trp} at the seventh step, we calculate
 \begin{align*}
 \|rxr\|&=\|rx^{1/2}\|^{2}=\|x^{1/2}rx^{1/2}\|\\
 & \geq \|x^{1/2}r_{1}x^{1/2}\|=\|r_{1}xr_{1}\| \\
 &\geq \|r_{1}f_{g_{0}}r_{1}xr_{1}f_{g_{0}}r_{1}\| \\
 & \geq \|f_{g_{0}}qxqf_{g_{0}}\|-2\delta \\
& >1-3\delta > 1-\varepsilon/2.
 \end{align*}
 This completes the proof of \eqref{it13.wtrp=trp}--\eqref{it16.wtrp=trp}.
 
 For each $g \in G$, let $(r_{g,k})_{k \in \N} \in \ell^{\infty}(A)$ 
 be a representing sequence for $r_{g}$, that is,
  $\pi_{\infty}(r_{g,1}, r_{g,2}, \cdots)=r_{g}$. 
  Since $r_{g}$ is a projection, we can assume that $r_{g,k}$ 
  is also a projection for all $g \in G$ and all $k \in \N$. 
  
  Choose $\eta>0$ such that if
  $a_{1},\ldots,a_{n}$ are positive
contractions in $A$ where $\|a_{i}^{2}-a_{i}\|<\eta$ and $\|a_{i}a_{j}\|<\eta$
for all $i,j=1,\ldots,n$ with $i\neq j$, then there are mutually orthogonal projections
$p_{1},\ldots,p_{n}$ in $A$ such that $\|a_{i}-p_{i}\|<\varepsilon/4n$ for all $i=1,\ldots,n$.  

 It follows from \eqref{it15.wtrp=trp} that there is $v \in A_{\infty}$ 
 such that $\|(y^{2}-yry-\varepsilon/2)_{+}-vxv^{*}\| <\varepsilon /4$, and so
 \begin{equation}\label{inq4.wtrp=trp}
 \|y^{2}-yry-vxv^{*}\| < 3\varepsilon/4. 
 \end{equation}
 Let $(v_{k})_{k \in \N}$ be a representing sequence for $v$. 
 We can choose $k$ large enough such that $\|r_{g,k}r_{h,k}\| <\eta$
  for all $g, h \in G$ with $g \neq h$ (since $r_{g}r_{h}=0$), and  that the following hold 
  (by \eqref{it13.wtrp=trp}, \eqref{it14.wtrp=trp}, \eqref{it16.wtrp=trp},
   and \eqref{inq4.wtrp=trp}):
 
 \begin{enumerate}
\setcounter{enumi}{\value{TmpEnumi}}
\item \label{it17.wtrp=trp}
$\|ar_{g,k}-r_{g,k}a\| <\varepsilon/2$ for all $g \in G$ and all $a \in F$;
\item \label{it18.wtrp=trp}
$\|\alpha_{g}(r_{h,k})-r_{gh,k}\| <\varepsilon /2$ for all $g, h \in G$;
\item \label{it19.wtrp=trp}
$\|y^{2}-yr_{k}y-v_{k}xv_{k}^{*}\|<3\varepsilon/4$, where $r_{k}=\sum_{g \in G} r_{g,k}$;
\item \label{it20.wtrp=trp}
$\|r_{k}xr_{k}\| > 1-\varepsilon/2$.

\setcounter{TmpEnumi}{\value{enumi}}
\end{enumerate}
 By the choice of $\eta$, there are orthogonal projections 
 $(p_{g})_{g \in G}$ in $A$ such that $\|r_{g,k}-p_{g}\| <\varepsilon/4n$ 
 for all $g \in G$. Since $\|r_{g,k}-p_{g}\| <\varepsilon/4$, \eqref{it17.wtrp=trp} implies 
 \eqref{it1.wtrp=trp}, and \eqref{it18.wtrp=trp} implies \eqref{it2.wtrp=trp}. 
 Put $p=\sum_{g \in G}p_{g}$. Then $\|r_{k}-p\| \leq \sum_{g \in G}\|r_{g,k}-p_{g}\| < \varepsilon/4$. 
 Then by \eqref{it19.wtrp=trp}, we have 
 $\|y^{2}-ypy-v_{k}xv_{k}^{*}\|<3\varepsilon/4+\varepsilon/4=\varepsilon$ 
 and so $(y^{2}-ypy-\varepsilon)_{+} \precsim_{A}  x$,
 which is \eqref{it3.wtrp=trp}. Finally, \eqref{it20.wtrp=trp} implies that $\|pxp\|>1-\varepsilon$, 
 which is \eqref{it4.wtrp=trp}. This finishes proof. 
  \end{proof}
 
 We present a list of examples of finite group actions with 
the (weak) tracial Rokhlin property on  simple nonunital C*-algebras. 
These examples are mainly based on
 the results of \cite{AGJP17} (and Theorems~\ref{proptensor} and \ref{thmtensor} below).
We refer the reader to \cite{AGJP17} for the proofs.

\begin{eg}\label{example_all}
In the following, $\W$ denotes the Razak-Jacelon  algebra \cite{Jac13}
and $S_{m}$ denotes the group of all 
permutations of $\{1,2,\ldots, m\}$, for $m\in\mathbb{N}$.

 \begin{enumerate}
 \item\label{example_all_it1}
 Let $A =\bigotimes _{k=1} ^{\infty}  {M}_{3}$ be the UHF~algebra of type $3^{\infty}$ and
let $B$ be a 
simple C*-algebra. 
Define $\alpha \colon \mathbb{Z}_{2} \to \mathrm{Aut}(A)$  by 
\[
\alpha=\bigotimes_{k=1}^{\infty} \mathrm{Ad} \begin{pmatrix}1&0&0 \cr 0&1&0 \cr 0&0&-1 \end{pmatrix}. 
\]
Then the action 
$ \alpha\otimes \mathrm{id} \colon \mathbb{Z}_{2} \rightarrow \mathrm{Aut}(A \otimes  B)$ 
has the weak tracial Rokhlin property, by Theorem~\ref{proptensor} 
and the the fact $\alpha$ has the tracial Rokhlin property.
(See \cite[Proposition~2.5]{Ph15} and \cite[Example~10.3.23 and Remark~10.4.9]{Ph17} 
for details about
$\alpha$.)
 In particular, if we take $B=\mathcal{W}$, 
then $ M_{3^{\infty}} \otimes \W \cong \W$ is  
stably projectionless and 
 $\alpha \otimes \id_{\W} \colon \Z_{2} \to \Aut(M_{3^{\infty}} \otimes \W)$ 
has the weak tracial Rokhlin property but not  
the tracial Rokhlin property since 
$ M_{3^{\infty}} \otimes \W$ does not have
any nonzero projections.
(We do not know whether this action has the Rokhlin property.)

\item\label{example_all_it2}
Let $A$ be a  nonelementary simple  
C*-algebra with tracial rank zero and 
let $m \in \N \setminus \{1\}$. Then the permutation action 
$\beta \colon S_{m} \rightarrow \Aut (A^{\otimes m})$ 
has the 
the  tracial Rokhlin property \cite{AGJP17}.
Here,  $A^{\otimes m}$ denotes the minimal tensor product of $m$ 
copies of $A$.
This result is similar to  
\cite[Example~5.10]{Hirshberg} which states that
the permutation action of $S_{m}$ on
$\mathcal{Z}^{\otimes m}\cong \mathcal{Z}$ 
has the 
generalized tracial Rokhlin property.
It is not clear that the action
$\beta$
 does not have 
the Rokhlin property.
It may depend on $A$. For example,
 if either $K_{0}(A)=\mathbb{Z}$ or  $K_{1}(A)=\mathbb{Z}$, then
  $\beta$ does not have the Rokhlin property (by \cite[Corollary~3.10]{Nawata}).

\item
The  flip action on
$ \mathcal{Z}\otimes  \mathcal{Z}\cong \mathcal{Z}$ has the weak tracial 
Rokhlin property  \cite[Example~5.10]{Hirshberg}
 but 
  not  the Rokhlin property. 
 For a nonunital example, if we  take $A=\mathcal{Z}\otimes \mathcal{K}$
 then the flip action on $A\otimes A \cong A$  has the weak
 tracial Rokhlin property (by a result of \cite{AGJP17}
 and that $A$ is tracially $\mathcal{Z}$-absorbing).
  This action does  not have the 
Rokhlin property, by \cite[Corollary~3.10]{Nawata}
since $K_{0}(A)=\mathbb{Z}$. We do not know whether
this action has the tracial Rokhlin property.

\item
Let $A$ be a simple 
  $\mathcal{Z}$-absorbing
 C*-algebra. Then for every finite nontrivial group $G$ there is an action
 $\alpha\colon G \to \mathrm{Aut}(A)$ with the
 weak tracial Rokhlin property. 
This follows essentially from  the fact that
$G$ embeds into some  
$S_{m}$ and that the permutation action
of $S_{m}$ on
$\mathcal{Z}^{\otimes m}\cong \mathcal{Z}$
has the
 weak tracial Rokhlin property.
 (One also needs   Theorem~\ref{proptensor} and Proposition~\ref{subgrp}).
 If moreover, $A$ is separable
 with either  $K_{0}(A)=\mathbb{Z}$ or  $K_{1}(A)=\mathbb{Z}$,
 then $\alpha$ does not have the Rokhlin property \cite[Corollary~3.10]{Nawata}.
 
 \item
 Let $\alpha$ be Blackadar's action of 
 $\mathbb{Z}_{2}$ on $M_{2^{\infty}}$ (see 
  \cite[Section~5]{Bla90} and \cite[Example~3.1]{Ph15}).
Then       
  $\alpha$ does not have the Rokhlin property but it
  has the tracial Rokhlin property 
\cite[Proposition~3.4]{Ph15}.
  Now consider the action
  $\alpha \otimes \id  \colon \Z_{2} \to M_{2^{\infty}}\otimes \mathcal{K}$.
 Then $\alpha \otimes \id$ has the weak tracial Rokhlin property
 (by Theorem~\ref{proptensor}) and so
    the  tracial Rokhlin property (by Theorem~\ref{wtr=tr}).
  However, $\alpha \otimes \id$ does not have the Rokhlin property
  (by \cite[Theorem~3.2(ii)]{Sa15}).
 \end{enumerate}
\end{eg}

\section{Permanence properties}\label{sec_per} 
The purpose of this section is to study the behavior of the (weak)
tracial Rokhlin property for finite group actions on simple C*-algebras
under restriction to subgroups, invariant hereditary subalgebras, 
 direct limits, and tensor products of actions.
 The following proposition is a nonunital version of \cite[Lemma~5.6]{Echterhoff-Phillips}. 

\begin{prop} \label{subgrp}
Let $\alpha \colon G \rightarrow \mathrm{Aut}(A)$ be an action of a finite group $G$ on a simple
 C*-algebra $A$ with the (weak) tracial Rokhlin property. If $H$ is a subgroup of $G$, then the restriction of 
 $\alpha$ to $H$ also has  the (weak) tracial Rokhlin property.
\end{prop}

\begin{proof}
The proof is similar to the unital case. The main idea is the following. 
Let $T$ be a set of right coset
 representations for $H$ in $G$.  If  $(f_{g})_{g \in G}$ is a suitable family
of Rokhlin elements satisfying Definition~\ref{w.t.R.p} for the action of $G$ on $A$, then
 we set 
 $e_{h}=\sum _{t \in T} f_{ht}$ for any $h \in H$.
 Then $(e_{h})_{h \in H}$ is a   family of Rokhlin elements  for the action
  of $H$ on $A$.
\end{proof}

Next, we show that the weak tracial Rokhlin property is preserved by 
passing to invariant hereditary subalgebras. 


\begin{prop} \label{propwtrpher}
Let $A$ be a simple C*-algebra and let $\alpha \colon G\to \mathrm{Aut}(A)$
be an action of a finite group $G$ on $A$.
Let $B$ be an $\alpha$-invariant hereditary C*-subalgebra of $A$
and let $\beta \colon G\to \mathrm{Aut}(B)$ be the restriction of $\alpha$ to $B$.
If $\alpha$ has the weak tracial Rokhlin property then
so does $\beta$.
If  $\alpha$ has the tracial Rokhlin property, then
so does  $\beta$ whenever 
 $B$  have an approximate
 identity  (not necessarily increasing)  consisting of projections
 in the fixed point algebra.
\end{prop}

\begin{proof}
  Suppose
  that $\alpha$ has the weak tracial Rokhlin property,
  and let  we are give a finite set $F \subseteq B$, $\varepsilon >0$, and
  $x, y \in B_{+}$ with $\|x\|=1$. We can assume that $x, y \in F$
  and that $F$ is contained in the closed unit ball of $B$.
  Without  loss of generality $\varepsilon<1$.
  Set $n=\card(G)$. 
   By   \cite[Lemma~2.5.12]{Lin.book}, there is $\delta >0$  such that 
   if $(e_{g})_{g \in G}$ 
 is a family of positive contractions in $B$ satisfying $\|e_{g}e_{h}\| < 2\delta$ 
 for all $g, h \in G$ with $g \neq h$, 
 then there are  orthogonal positive contractions $(f_{g})_{g \in G}$ 
 in $B$ such that 
 $\|f_{g}-e_{g}\| < \tfrac{\varepsilon}{4n}$ for any $g \in G$. 
 We may assume  that $\delta < \tfrac{\varepsilon}{28n}$.
  
Since any approximate identity for $B^{\beta}$ is also an approximate identity for
 $B$, we can choose a positive contraction $b\in B^{\beta}$ such that 
 \begin{equation}\label{app.herd}
 \|ab-a\|<\delta \ \ \text{and} \ \ \|ba-a\| <\delta,\  
  \end{equation}
for any $a \in F$. We set $z=bxb$ and $w=byb$. 
So by \eqref{app.herd} we have
\begin{equation}\label{eq1_wtrp.herd}
\|z-x\|<2\delta\ \ \text{and}\  \   \|w-y\|<2\delta .
\end{equation} 
Applying Definition~\ref{w.t.R.p} to $\alpha$ with $F \cup \{b\}$ in place of $F$, 
with $z/\|z\|$ in place of $x$,
with $w$ in place of $y$, and  with $\delta$ in place of $\varepsilon$, we  obtain 
orthogonal positive contractions $(r_{g})_{g \in G}$ in $A$ such that, 
with $r=\sum_{g \in G} r_{g}$, the following hold:

\begin{enumerate} 
\item \label{it1.wtrp.herd}
$\|r_{g}a -ar_{g}\| < \delta$ for all $a \in F \cup \{b\}$ and all $g \in G$;
\item \label{it2.wtrp.herd}
$\|\alpha_{g}(r_{h})-r_{gh}\| < \delta$ for all $g, h \in G$;
\item \label{it3.wtrp.herd}
$(w^{2}-wrw-\delta)_{+} \precsim_{A} z$;
\item \label{it4.wtrp.herd}
$\|rzr\| > \|z\| (1-\delta)$.
 \setcounter{TmpEnumi}{\value{enumi}}
 \end{enumerate}
We put $e_{g}=br_{g}b$ for any $g \in G$, and $e=\sum_{g \in G} e_{g}$.
For any $g,h\in G$ with $g\neq h$,
using \eqref{it1.wtrp.herd} and that $r_{g}r_{h}=0$ at the third step,
we have
\[
\|e_{g}e_{h}\|=\|br_{g}b^{2}r_{h}b\|\leq \| r_{g}b^{2}r_{h}\|< 2\delta.
\]
Hence by the choice of $\delta$ there are
orthogonal positive contractions $(f_{g})_{g \in G}$ 
 in $B$ such that 
 $\|f_{g}-e_{g}\| < \tfrac{\varepsilon}{4n}$ for any $g \in G$. 
 We set $f=\sum_{g \in G} f_{g}$. Then we have
 \begin{equation}\label{eq2_wtrp.herd}
 \|f_{g}-e_{g}\| < \tfrac{\varepsilon}{4}\ \ \text{and}\ \ \|f-e\|<\tfrac{\varepsilon}{4}.
 \end{equation}
In particular, $\|e\|<\|f\|+\tfrac{\varepsilon}{4}<2$.
 We show that the family $(f_{g})_{g \in G}$ satisfies the conditions of
 Definition~\ref{w.t.R.p} for the action $\beta \colon G\to \mathrm{Aut}(B)$,
 that is, we will show that the following hold:
 \begin{enumerate}
 \setcounter{enumi}{\value{TmpEnumi}}
 
\item\label{it5.wtrp.herd}
 $\|f_{g}a-af_{g}\| < \varepsilon$ for all $a\in F$ and all $g\in G$;
\item\label{it6.wtrp.herd}
 $\|\beta_{g}(f_{h})-f_{gh}\| < \varepsilon$ for all $g,h\in G$;
\item\label{it7.wtrp.herd}
  $(y^{2}-yfy - \varepsilon)_{+} \precsim_{B} x$;
\item\label{it8.wtrp.herd}
  $\|fxf\| > 1-\varepsilon$.
  
\setcounter{TmpEnumi}{\value{enumi}}
\end{enumerate}
To see \eqref{it5.wtrp.herd}, using \eqref{eq2_wtrp.herd}, \eqref{app.herd}, and \eqref{it1.wtrp.herd},
for any $a\in F$ and any $g\in G$ we get
\begin{align*}
\|f_{g}a-af_{g}\| &\leq \|f_{g}a-e_{g}a\|+\|e_{g}a-br_{g}ab\|+\|br_{g}ab-bar_{g}b\|\\
&\ \ \ +\|bar_{g}b - ae_{g}\|+ \|ae_{g} -af_{g}\|\\
&<\tfrac{\varepsilon}{4} +2\delta +\delta+ 2\delta+ \tfrac{\varepsilon}{4}<\varepsilon.
\end{align*}

To prove \eqref{it6.wtrp.herd}, using \eqref{eq2_wtrp.herd} and \eqref{it2.wtrp.herd}, 
for any $g,h\in G$ we have
\begin{align*}
\|\beta_{g}(f_{h})-f_{gh}\| &\leq \|\alpha_{g}(f_{h})-\alpha_{g}(e_{h})\|
+ \|b\alpha_{g}(r_{h})b-br_{gh}b\|
+\|e_{gh}- f_{gh}\|\\
& <\tfrac{\varepsilon}{4} + \delta + \tfrac{\varepsilon}{4}<\varepsilon.
\end{align*}

To see  \eqref{it7.wtrp.herd}, first using \eqref{eq1_wtrp.herd} at the second step
and using \eqref{app.herd} and \eqref{eq2_wtrp.herd} at the {f}ifth step, we obtain
\begin{align*}
\left\|(y^{2}-yfy)-(w^{2}-wrw-\delta)_{+} \right\|&\leq \delta+
\|y^{2}-w^{2}\|+\|yfy-wrw\|\\
&< \delta +4 \delta + \|yfy-byeyb\|\\
&\leq  5\delta + \|yfy-byfy\|+\|byfy-byey\|\\
&\ \ \ +\|byey- byeyb\|\\
&< 5\delta + \|y-by\| +\| f-e\| +\|e\|\cdot\|y-yb\|\\
&<5\delta + \delta +\tfrac{\varepsilon}{4} +2\delta=8\delta+ \tfrac{\varepsilon}{4}
<\tfrac{\varepsilon}{3}+\tfrac{\varepsilon}{4}<\varepsilon.
\end{align*}
Now, using Lemma~\ref{lemkr} and \eqref{it3.wtrp.herd}, we get
\[
(y^{2}-yfy - \varepsilon)_{+} \precsim_{A} (w^{2}-wrw-\delta)_{+} \precsim_{A} z \precsim_{A} x.
\]
 Since $B$ is a hereditary subalgebra of $A$, we 
get  $(y^{2}-yfy - \varepsilon)_{+}\precsim_{B} x$,
 which is \eqref{it7.wtrp.herd}.
To show \eqref{it8.wtrp.herd}, 
first by \eqref{eq1_wtrp.herd} we have $\|z\|>1-2\delta$, and so
by \eqref{it4.wtrp.herd} we get
\begin{equation}\label{eq3_wtrp.herd}
\|rzr\|>\|z\|(1-\delta)>(1-2\delta)(1-\delta)>1-3\delta.
\end{equation}
Second,
using \eqref{app.herd} at the third step,
 using \eqref{it1.wtrp.herd} at the fourth and fifth steps,
 and using \eqref{eq3_wtrp.herd} at the seventh step, we obtain
\begin{align}\label{eq4_wtrp.herd}
 \|exe\| &= \|brbxbrb\|\\ \notag
&\geq \|brxrb\| -\|br(bxb-x)rb\|\\ \notag
&>\|brxrb\| -2\delta\\\notag
&>\|rbxrb\|-n\delta -2\delta\\ \notag
&>\|rbxbr\|-2n\delta -2\delta\\ \notag
&=\|rzr\|-(2n+2)\delta\\ \notag
&>1-3\delta-(2n+2)\delta  \notag\\ 
&=1-(2n+5)\delta\geq 1-\tfrac{\varepsilon}{4}  \notag.
\end{align}
Then using \eqref{eq2_wtrp.herd} at the third step and using
\eqref{eq4_wtrp.herd} at the fourth step, we get
\begin{align*}
\|fxf\| &\geq \|exe\| - \|exe-exf\|- \|exf-fxf\|\\ 
&\geq \|exe\| - \|e\|\cdot \|e-f\|
-\|e-f\| \\
&>\|exe\| - \tfrac{3\varepsilon}{4} \\
&> 1- \tfrac{\varepsilon}{4} - \tfrac{3\varepsilon}{4}=1-\varepsilon,
\end{align*}
which is \eqref{it8.wtrp.herd}.
This completes the proof of \eqref{it5.wtrp.herd}--\eqref{it8.wtrp.herd},
and shows that 
$\beta \colon G\to \mathrm{Aut}(B)$  has  the weak tracial Rokhlin property.

The proof of the second part of the statement about the tracial Rokhlin property
is similar to the proof of the first part.
 \end{proof}
 
\begin{prop}\label{corindlim}
Let $G$ be a finite group. Let 
$\left( (G,A_{i},\alpha^{(i)})_{i\in I}, (\varphi_{j,i})_{i\leq j}\right)$ be a direct system 
of simple $G$-algebras. Let $A$ be the direct limit of the $A_{i}$ and let
$\alpha \colon G\to \mathrm{Aut}(A)$ be the direct limit of the $\alpha^{(i)}$.
If  $\alpha^{(i)}$ has the (weak) tracial Rokhlin property for each $i$, then
so does $\alpha$.
\end{prop}

\begin{proof}
The statement follows essentially from the following fact.
If  $\alpha \colon G\to \mathrm{Aut}(A)$
is an action of $G$  
on a simple C*-algebra $A$ such that
for every finite set $F \subseteq A$ and every $\varepsilon>0$ there is
an $\alpha$-invariant simple C*-subalgebra $B$ of $A$ such that
$F \subseteq_{\varepsilon} B$ and the restriction of $\alpha$ to $B$ has
the
(weak)   tracial Rokhlin property,
then $\alpha \colon G\to \mathrm{Aut}(A)$ has the 
(weak)
 tracial Rokhlin property. 
\end{proof}

 Phillips  posed the following problem 
for actions on  simple unital C*-algebras. 

\begin{pbm}[\cite{Ph.free}, Problem~3.18]\label{prob_3.18}
 Let $A$ and $B$ be infinite dimensional simple unital 
C*-algebras, and let $\alpha \colon G \rightarrow \Aut(A)$ be an action 
with the tracial Rokhlin property and $\beta \colon G \rightarrow \Aut(B)$ 
be an arbitrary action. Does it follow that 
$\alpha \otimes \beta \colon G \rightarrow \Aut(A \otimes_{\min} B)$ 
has the tracial Rokhlin property? 
\end{pbm}

There are some partial solutions to this problem. 
Lemma~3.9 of \cite{Ph11} is the very special case $B = M_n$ and $\beta$ is inner. 
If $A$ has tracial rank zero and $B$ has tracial rank at most one, then by
\cite[Proposition~3.19]{Ph.free},  $\alpha \otimes \beta$ has the tracial 
Rokhlin property. Moreover, it follows from~\cite[Proposition~2.4.6]{Wang13} 
that this problem has an affirmative answer provided that $A \otimes_{\min} B$ has 
Property~(SP). The following two results are  more general 
 solutions to this problem in the \emph{not necessarily unital} simple case.

\begin{thm}\label{proptensor}
Let $\alpha \colon G\to \mathrm{Aut}(A)$ and $\beta \colon G\to \mathrm{Aut}(B)$
be actions of a finite group $G$ on simple C*-algebras $A$ and $B$.
If $\alpha$ has the weak tracial Rokhlin property
then so does $\alpha \otimes \beta \colon G \to \mathrm{Aut}(A \otimes_{\min} B)$.
\end{thm}
\begin{proof}
Suppose that we are given a finite set $F \subseteq A \otimes_{\min} B$, $\varepsilon >0$, and 
$x, y \in (A\otimes_{\min} B)_{+}$  with $\|x\|=1$. 
Set $D=A \otimes_{\min} B$.
We shall find  a family of orthogonal positive contractions 
$(f_{g})_{g \in G}$ in $A \otimes_{\min} B$ such that, with 
$f= \sum _{g \in G} f_{g}$, the following hold:
\begin{enumerate}
\item\label{proptensor_it1}
 $\|f_{g}a-af_{g}\| < \varepsilon$ for all $a\in F$ and all $g\in G$;

\item\label{proptensor_it2}
$\|\alpha_{g}(f_{h})-f_{gh}\| < \varepsilon$ for all $g,h\in G$;

\item\label{proptensor_it3}
$(y^{2}-yfy - \varepsilon)_{+} \precsim_{D} x$;

\item\label{proptensor_it4}
$\|fxf\| > 1-\varepsilon$.

\setcounter{TmpEnumi}{\value{enumi}}
\end{enumerate}
We may assume that  there exist  $c_{1},\ldots,c_{m}$ in $A$ and 
$d_{1},\ldots,d_{m}$ in $B$ such that 
$F=\{ c_{i}\otimes d_{i} \mid 1\leq i \leq m\}$ and that 
 $\|c_{i}\|, \|d_{i}\| \leq 1$ for all $1\leq i \leq m$.
By Part~\eqref{rmkdefwtrp_it3} of Remark~\ref{rmkdefwtrp}, 
we may assume that $y=y_{1} \otimes y_{2}$ for some 
$y_{1} \in A_{+}$ and $y_{2} \in B_{+}$ with 
$\| y_{1} \|, \| y_{2}\| \leq 1$. 
Choose $\delta$ such that $0 < \delta < \tfrac{\varepsilon}{3}$ and 
$\left((1-\delta)^{2}-4 \delta-\delta^{2}\right) > 1-\varepsilon$.
There exists $\delta_{1}$ such that  $\tfrac{1}{2} < \delta_{1} < 1$ 
and  $\left((1-\delta)^{2}-4 \delta-\delta^{2}\right)\delta_{1}^{2} > 1-\varepsilon$.
Put $z=(x-\delta _{1})_{+}$. It follows from Kirchberg's Slice Lemma 
 (\cite[Lemma~4.1.9]{Ro02}) that there are elements 
 $x_{1} \in A_{+}$ and $x_{2} \in B_{+}$ such that 
 $x_{1} \otimes x_{2} \precsim_{D} z$  and that 
 $\|x_{1}\|=\|x_{2}\|=1$. By Lemma~\ref{lembdd1}, there exists $ w \in A \otimes_{\min} B$ 
 such that $\|wxw^{*}-x_{1} \otimes x_{2}\| <\delta^{2}$ and 
 $\|w\| \leq \delta_{1} ^{-{1}/{2}} < \delta_{1}^{-1}$.  
 Thus there is $ v \in A \otimes_{\min} B$  where 
 $v= \sum _{i=1} ^{k} v_{i} \otimes w_{i}$ for some $ v_{i} \in A$ and $w_{i} \in B$, $i=1,\ldots,k$,
  such that
\begin{enumerate}
\setcounter{enumi}{\value{TmpEnumi}}

\item\label{proptensor_it5} 
$\|vxv^{*}-x_{1}\otimes x_{2}\| < \delta^{2}$ and $\|v\| < \delta_{1}^{-1}$.

\setcounter{TmpEnumi}{\value{enumi}}
\end{enumerate}
Put $E= \{c_{i} \mid 1 \leq i \leq m \} \cup \{v_{i} \mid 1 \leq i \leq k\} $. 
 By \cite[Proposition~2.7(v)]{KR00}, there is $ n \in \mathbb{N}$  such that
\begin{enumerate}
\setcounter{enumi}{\value{TmpEnumi}}

\item\label{proptensor_it6} 
$(y^{2}_{2}-\delta)_{+} \precsim_{B} x_{2} \otimes 1_{n}$. 

\setcounter{TmpEnumi}{\value{enumi}}
\end{enumerate}
By Proposition~\ref{outer}, $\alpha$ is pointwise outer, and so  $A$ is not elementary. 
It follows from \cite[Corollary~IV.1.2.6]{Bl06} 
that $A$ is not of Type~I.  
Now, \cite[Lemma~2.4]{Ph14} implies that there is a nonzero element $x_{0} \in A_{+}$ such that
\begin{enumerate}
\setcounter{enumi}{\value{TmpEnumi}}

\item\label{proptensor_it7}
 $x_{0} \otimes 1_{n} \precsim_{A} x_{1}$.

\setcounter{TmpEnumi}{\value{enumi}}
\end{enumerate} 
Put $M=1+ \sum _{i=1} ^{k} \|v_{i}\| + \sum_{i=1} ^{k} \|w_{i}\|$. 
Choose $\eta >0$ such that $\eta < \tfrac{\delta}{2Mk \card(G)}$. 
Applying  Lemma~\ref{lemtrp2} to the action $\alpha$ with $E$ in place of $F$,
with $\eta$ in place of $\varepsilon$, with $x_{0}$ in place of $x$, 
with $y_{1}$ in place of $y$, 
and with $x_{1}$ in place of $z$, we obtain  a family of  orthogonal 
positive contractions $(r_{g})_{g \in G}$ in $A$ 
such that, with $r=\sum _{g \in G} r_{g}$, the following hold:
\begin{enumerate}
\setcounter{enumi}{\value{TmpEnumi}}

\item\label{proptensor_it8}
$\|r_{g}c-cr_{g}\| < \eta$ for all $c\in E$ and all $g\in G$;

\item\label{proptensor_it9} 
$\|\alpha_{g}(r_{h})-r_{gh}\| < \eta$ for all $g,h\in G$;

\item\label{proptensor_it10} 
$(y_{1}^{2}-y_{1}ry_{1} - \eta)_{+} \precsim_{A} x_{0}$;

\item\label{proptensor_it11} 
$\|rx_{1}r\| > 1-\eta$.

\setcounter{TmpEnumi}{\value{enumi}}
\end{enumerate}

On the other hand, since $B$ has an approximate identity contained in $B^{\beta}$,
we can choose a positive contraction $s \in B^{\beta}$ such that
\begin{enumerate}
\setcounter{enumi}{\value{TmpEnumi}}

\item\label{proptensor_it12}
$\|y_{2}sy_{2}-y_{2}^{2}\| < \eta$, $\|[s,d_{i}]\| < \eta$ 
for all $ 1 \leq i \leq m$, $\| [s, w_{j}] \| < \eta$ for all $ 1 \leq j \leq k$, 
and $ \|sx_{2}s\| > 1- \eta$.

\setcounter{TmpEnumi}{\value{enumi}}
\end{enumerate}
Put $f_{g}= r_{g} \otimes s$ for all $g \in G$, and put $f =\sum _{g \in G} f_{g}$. 
Then $(f_{g})_{g \in G}$ is a family of mutually orthogonal positive contractions in $A \otimes_{\min} B$.  
We show that \eqref{proptensor_it1}--\eqref{proptensor_it4} hold. 
For \eqref{proptensor_it1}, let $1\leq i \leq m$. Then by  
\eqref{proptensor_it8} and \eqref{proptensor_it12} we have
\begin{align*}
\|[f_{g}, c_{i}\otimes d_{i}]\| & = \|[r_{g}\otimes s, c_{i}\otimes d_{i}]\|\\
&=\|(r_{g}c_{i}) \otimes (sd_{i})- (c_{i}r_{g}) \otimes (d_{i}s)\|\\
&\leq \|[r_{g},c_{i}] \otimes (sd_{i})\|+
\|(c_{i}r_{g}) [s,d_{i}]\|\\
&<\eta+\eta <2\delta <\varepsilon.
\end{align*}
Part~\eqref{proptensor_it2} follows from \eqref{proptensor_it9}. 
To prove \eqref{proptensor_it3}, first using \eqref{proptensor_it12} at the third step we have
\begin{align*}
\| \left(y^{2}-yfy\right)&-\left(y_{1}^{2}-y_{1}ry_{1}-\eta\right)_{+} \otimes \left(y_{2}^{2}-\delta\right)_{+}\|\\
&\leq \|\left(y_{1}^{2} \otimes y^{2}_{2}\right)-\left(y_{1}ry_{1}\right) \otimes \left(y_{2}sy_{2}\right)-
\left(y_{1}^{2}-y_{1}ry_{1}\right) 
\otimes y_{2}^{2}\|+\eta+ \delta\\
&= \|\left(y_{1}ry_{1}\right) \otimes \left(y_{2}sy_{2}-y_{2}^{2}\right)\|+2\delta\\
&< \delta + 2 \delta < \varepsilon.
\end{align*}
 Then by  Lemma~\ref{lemkr} at the first step, by 
 \eqref{proptensor_it6} and \eqref{proptensor_it10} at the second step, and by 
 \eqref{proptensor_it7} 
 at the fourth step, we  get
 \begin{align*} 
\left(y^{2}-yfy-\varepsilon\right)_{+} 
&\precsim_{D}   \left(y_{1}^{2}-y_{1}ry_{1}- \eta\right)_{+} \otimes \left(y^{2}_{2}- \delta\right)_{+} \\
& \precsim_{D} 
 x_{0} \otimes (x_{2} \otimes 1_{n}) \\  
 &  \sim_{D} (x_{0} \otimes 1_{n}) \otimes  x_{2} \\
& \precsim_{D} x_{1} \otimes x_{2} \precsim_{D} z
 \precsim_{D} x,
\end{align*}
which is \eqref{proptensor_it3}.
To prove \eqref{proptensor_it4}, first by \eqref{proptensor_it8}  and 
\eqref{proptensor_it12} we have 
\begin{align*}
\|fv-vf\| &= \left\| \sum_{i=1}^{k}(rv_{i}) \otimes (sw_{i})-  
\sum_{i=1}^{k}(v_{i}r) \otimes (w_{i}s) \right\|\\
&\leq \sum_{i=1}^{k}\| (rv_{i}-v_{i}r) \otimes (sw_{i})\|+
\sum_{i=1}^{k}\|(v_{i}r) \otimes (sw_{i}-w_{i}s) \|\\
&\leq M \sum_{i=1}^{k} \sum_{g\in G}\|r_{g}v_{i}-v_{i}r_{g}\|+ 
M \sum_{i=1}^{k} \|sw_{i}-w_{i}s\|\\
&< Mk\, \card(G) \eta + Mk \eta \leq 2Mk\, \card(G) \eta < \delta.
\end{align*}

Also, by \eqref{proptensor_it11} and \eqref{proptensor_it12} we get
\begin{align*}
\|f (x_{1} \otimes x_{2}) f\| &= \|(rx_{1}r) \otimes (s x_{2} s)\|= \|rx_{1}r\| \cdot \|s x_{2} s\|\\
&> (1-\delta)(1-\eta) >(1-\delta)^{2}.
\end{align*}

Then by using these two latter inequalities and \eqref{proptensor_it5} we calculate 
\begin{align*}
(1-\delta)^{2} &< \|f(x_{1} \otimes x_{2})f \| \\
&  \leq \|fvxv^{*}f\|+\|f(vxv^{*}-x_{1} \otimes x_{2})f\| \\
& < \|vfxv^{*}f\|+\|(fv-vf)xv^{*}f\|+ \delta ^{2} \\
& \leq \|vfxfv^{*} \|+\|vfx(v^{*}f-fv^{*})\|+\delta \|v\|+ \delta ^{2}\\
& \leq  \|v\|^{2} \|fxf\|+ \delta \|v\|+ \delta \|v\|+ \delta ^{2} \\
& \leq \delta _{1} ^{-2} \|fxf\|+2 \delta \delta_{1}^{-1}+ \delta^{2} \\
& < \delta_{1} ^{-2}\|fxf\|+4 \delta+ \delta^{2}.
\end{align*}
Therefore, by the choice of $\delta_{1}$ we obtain
\[
\|fxf\| > \left( (1-\delta)^{2}-4 \delta-\delta^{2}\right)\delta_{1}^{2} > 1-\varepsilon
\]
which is \eqref{proptensor_it4}. This completes the proof.
\end{proof}

The proof of the following result is similar to that of the preceding theorem.

\begin{thm}\label{thmtensor}
Let $\alpha \colon G\to \mathrm{Aut}(A)$ and $\beta \colon G\to \mathrm{Aut}(B)$
be actions of a finite group $G$ on simple C*-algebras $A$ and $B$.
Let $\alpha$ have the tracial Rokhlin property and let $B^{\beta}$ have an
approximate identity (not necessarily increasing) consisting of projections. 
Then the action $\alpha \otimes \beta \colon G \to \mathrm{Aut}(A \otimes_{\min} B)$
has the tracial Rokhlin property. In particular, Problem~\ref{prob_3.18}
has an affirmative answer.
\end{thm}

The following corollary follows from  Theorems~\ref{proptensor} and \ref{thmtensor} by taking
$B=M_{n}$ and $\beta$ to be the trivial action.

\begin{cor}\label{cortensor}
Let $\alpha \colon G\to \mathrm{Aut}(A)$ 
be an action of a finite  group $G$ with the (weak) tracial Rokhlin property 
on a simple  C*-algebra $A$.
Then the induced  action of $G$ on $M_{n}(A)$
has the (weak) tracial Rokhlin property for any $n\in \mathbb{N}$.
\end{cor}

The following result gives a criterion for the nonunital tracial Rokhlin property
in terms of the unital tracial Rokhlin property 
in the case that the underlying algebra has ``enough" projections.

\begin{prop}\label{proploctrp}
Let  $\alpha \colon G\to \mathrm{Aut}(A)$
be an action of a finite group $G$ on a simple C*-algebra $A$. Suppose that 
$A$ has an approximate identity (not necessarily increasing) $(p_{i})_{i\in I}$
consisting of projections such that each $p_{i}$ is in $A^{\alpha}$.
Then $\alpha$ has the tracial Rokhlin property
if and only if the restriction of $\alpha$ to $p_{i}Ap_{i}$ 
has the tracial Rokhlin property for every $i \in I$.
\end{prop}

\begin{proof}
The ``{if}" part follows from  Proposition~\ref{corindlim}, and the ``{only if}" part 
follows from the second part of  Proposition~\ref{propwtrpher}.
\end{proof}



\section{Crossed products}\label{sec_cross}
The main goal of this section is to 
show that some classes of simple C*-algebras are closed under
 taking crossed products 
and fixed point algebras by
actions of finite groups with the  tracial Rokhlin property.
In particular, this is true for the class of simple C*-algebras with tracial rank zero.
This  extends a result of Phillips (\cite[Theorem~2.6]{Ph11}) to the nonunital case
and is  evidence that our definition of the (weak) tracial Rokhlin property
on simple  C*-algebras is the right one.

The following proposition is essential
in the sequel.

\begin{prop}\label{propcross}
Let $G$ be a finite group and let $\mathcal{C}$ be a class of simple 
 C*-algebras with
the following properties:
\begin{enumerate}
\item\label{propcross_1}
if $A$ is a simple   C*-algebra and $p\in A$ is a nonzero projection, then
$A\in\mathcal{C}$ if and only if $pAp\in\mathcal{C}$ (in particular, this is the case
if $\mathcal{C}$ is closed under Morita equivalence);
\item\label{propcross_2}
if $A\in\mathcal{C}$ is unital and $\alpha$ is an action of $G$ on $A$ with the tracial
Rokhlin property then $A\rtimes_{\alpha} G\in \mathcal{C}$;
\item\label{propcross_3}
if $A\in \mathcal{C}$ and $B$ is a C*-algebra with $A \cong B$, then 
$B\in \mathcal{C}$.
\end{enumerate}
Then $\mathcal{C}$ is closed under  taking crossed products
and fixed point algebras by actions of $G$ with the
tracial Rokhlin property (and hence \eqref{propcross_2} above holds
without the  assumption that $A$ is unital).
\end{prop}

\begin{proof}
Let  $\mathcal{C}$ be a class of simple C*-algebras as in the statement.
Let $A\in \mathcal{C}$ and  let $\alpha \colon G \to \mathrm{Aut}(A)$
be an action with the tracial Rokhlin property.
We show that $A\rtimes_{\alpha} G\in \mathcal{C}$.
We may assume that  $A$ is nonzero.
First note that 
  there exists a nonzero projection $p$ in $A^{\alpha}$. 
In fact, if $(p_{g})_{g\in G}$
  is a family of Rokhlin projections for a given $\varepsilon$ according to
  Definition~\ref{w.t.R.p} and if $q=\sum_{g\in G}p_{g}$,
  then $\|E(q)-q\|<\varepsilon\,\mathrm{card}(G)$,
  where $E\colon A\to A^{\alpha}$ is the canonical conditional
  expectation. Using functional
calculus within $A^{\alpha}$ gives a projection $p\in A^{\alpha}$ arbitrarily close to $E(q)$.
So, $p$ is arbitrary close to $q$, and hence it is nonzero.
 
Let $\alpha \colon G\to \mathrm{Aut}(pAp)$ be the restriction of $\alpha$ to $pAp$.
Now, the second part  of Proposition~\ref{propwtrpher}
  implies that $\beta$ has the tracial Rokhlin property.
By Condition~\eqref{propcross_1}, $pAp \in \mathcal{C}$. Thus,
by Condition~\eqref{propcross_2},
 $pAp \rtimes _{\beta} G \in \mathcal{C}$. Observe that 
$pAp \rtimes _{\beta} G\cong p(A \rtimes _{\alpha} G)p$.
In fact, the map $\varphi\colon pAp \rtimes _{\beta} G\to p(A \rtimes _{\alpha} G)p$
defined by $\varphi(\sum_{g\in G}b_{g}u_{g})=\sum_{g\in G}b_{g}\delta_{g}$
where $b_{g}\in pAp$ for all $g\in G$, is easily seen to be a surjective $*$-isomorphism.
Thus, by Condition~\eqref{propcross_3}, $p(A \rtimes _{\alpha} G)p \in \mathcal{C}$.
Now Condition~\eqref{propcross_1} implies that 
$ A \rtimes _{\alpha} G \in \mathcal{C}$.
Also, $ A^{\alpha} \in \mathcal{C}$, 
by Condition~\eqref{propcross_1}.
\end{proof}

In the following theorem, we extend Phillips's
result \cite[Theorem~2.6]{Ph11} to the nonunital case. 

\begin{thm}\label{thmcross}
Let $A$ be a simple   C*-algebra with tracial rank zero and let $\alpha$
be an action of a finite group $G$ on $A$ with the weak tracial Rokhlin property.
Then the crossed product  $A \rtimes _{\alpha} G$ and the fixed point algebra
$A^{\alpha}$ are simple C*-algebras with tracial rank zero.
\end{thm}

\begin{proof}
 It follows from Theorem~\ref{wtr=tr} that  $\alpha$  in fact has the tracial Rokhlin property.
Let $\mathcal{C}$ denote the class of simple C*-algebras
with tracial rank zero. By 
Theorem~\ref{cormorita}, $\mathcal{C}$ satisfies Condition~\eqref{propcross_1} in
Proposition~\ref{propcross}. Also, by \cite[Theorem~2.6]{Ph11},
$\mathcal{C}$ satisfies Condition~\eqref{propcross_2} in
Proposition~\ref{propcross}. (Note that the assumption of separability
is unnecessary in \cite[Theorem~2.6]{Ph11}.)
Clearly, $\mathcal{C}$ satisfies Condition~\eqref{propcross_3} in
Proposition~\ref{propcross}.
Thus  Proposition~\ref{propcross} yields the first part of the statement
about the crossed product.

The second part of the statement about the fixed point algebra
follows from Corollary~\ref{simpleweak} which says that 
$A^{\alpha}$
 is isomorphic  to a full corner of
 $A \rtimes _{\alpha} G$, and 
 Theorem~\ref{cormorita} which implies that
 the tracial rank zero is passed to corners.
\end{proof}


The following corollary is immediate from Example~\ref{example_all}\eqref{example_all_it2}
 and Theorem~\ref{thmcross}.

\begin{cor}\label{cor_corss_tr0}
Let $A$ be a  simple 
nonelementary  C*-algebra with tracial rank zero and 
let $\beta \colon S_{m} \rightarrow A^{\otimes m}$ be the permutation 
action,
where $m\geq 2$.
Then the crossed product $A^{\otimes m} \rtimes _{\beta} S_{m}$ is a 
simple C*-algebra with  tracial rank zero.
\end{cor}


\begin{thm}\label{thmclass1}
The  class of simple separable nuclear $\mathcal{Z}$-absorbing   C*-algebras 
is preserved under taking crossed products 
and fixed point algebras by
finite group actions with the tracial Rokhlin property.
\end{thm}

\begin{proof}
Let $\mathcal{C}$ denote the class of simple separable nuclear $\mathcal{Z}$-absorbing C*-algebras.
By  \cite[Corollary~3.2]{TW07},  $\mathcal{Z}$-stability is preserved under  Morita equivalence 
in the class of separable C*-algebras. Moreover, by  \cite[Theorem~3.15]{HRW07}, nuclearity is 
preserved under  Morita equivalence. Thus the class $\mathcal{C}$ 
satisfies \eqref{propcross_1} in Proposition~\ref{propcross}.
On the other hand,   \cite[Corollary~5.7]{Hirshberg} implies that
 the class $\mathcal{C}$ 
also satisfies \eqref{propcross_2} in Proposition~\ref{propcross}.
Therefore, by Proposition~\ref{propcross} the class  $\mathcal{C}$  
is preserved under taking crossed products 
by finite group actions of with the tracial Rokhlin property.
The corresponding result about  $A^{\alpha}$ holds  by noting that 
 $A^{\alpha}$ is a full corner of
 $A\rtimes _{\alpha} G$.
\end{proof}



\begin{rem}\label{rem_class}
There are some other classes of simple C*-algebras which are preserved
under  taking crossed products and fixed point algebras
by finite group actions  with the weak tracial Rokhlin property.
For example, the class of simple purely infinite C*-algebras
(by \cite[Theorem~3]{Je95}), the class of simple C*-algebras with Property~(SP)
(by the nonunital version of \cite[Theorem~4.2]{JO98}),
and the class of simple tracially $\mathcal{Z}$-absorbing
C*-algebras \cite{AGJP17}.  
\end{rem}

\appendix

\section{C*-algebras with  tracial rank zero}\label{app}

We begin this section with recalling the definition  of tracial rank
zero for C*-algebras from \cite{Lin01}. Then, we give a characterization of
 tracial rank zero which unifies the definitions
  for the simple unital   and simple nonunital 
cases (Theorem~\ref{main.trk}).
The main advantage of this definition is to avoid working with 
the unitization of simple C*-algebras. 
In particular, we are able to show that having tracial rank zero
is preserved under Morita equivalence in the class of 
 simple C*-algebras.

\subsection{Preliminaries}
In this subsection  we present some notation 
and results which will be used in the sequel. 
Also, the statement of the main theorem of
this appendix (Theorem~\ref{main.trk}) is given
at the end of this subsection. The proof 
of this theorem will be given after
 Lemma~\ref{lemmatrixk}.
\begin{nota}
We  recall some notation  from \cite{Lin01} 
 for the convenience
of the reader. We remark that, instead of notation 
$[a]\leq [b]$ used in \cite{Lin01},
we will adopt the notation
 $a\precsim_{s} b$ from
 \cite[Definition~2.1]{ORT11}.  

\begin{enumerate}
\item
We denote by $\mathcal{I}^{(0)}$ the class of all finite dimensional C*-algebras.

\item 
Let $\sigma_{1},\sigma_{2}$ be real numbers with $0 < \sigma_{1} < \sigma _{2} \leq 1$. 
Define a continuous function 
$f_{\sigma_{1}}^{\sigma_{2}}\colon[0,\infty)\to [0,1]$ by 
\[
f_{\sigma_{1}}^{\sigma_{2}}(t)=
\begin{cases}
0  & 0\leq t< \sigma_{1},\\
\mathrm{linear} &  \sigma_{1} \leq t < \sigma_{2},\\
1 & t\geq \sigma_{2}.
\end{cases}
\]

\item 
Let $a$ and $b$ be  positive elements in a C*-algebra $A$. 
We say that $a$ is \emph{Blackadar subequivalent} to
$b$ and we write $a \precsim_{s} b$ if there exists 
$x \in A$ such that $x^{*}x=a$ and $xx^{*} \in \overline{bAb}$.
Note that, $a \precsim_{s} b$ is equivalent to
the relation $[a]\leq [b]$ which is used in
  \cite[Definition~2.2]{Lin01}  
(see   \cite[Section~4]{ORT11}).
Let $n$ be a positive integer. We write $a \precsim_{s,n} b$ if there are $n$ mutually orthogonal positive elements 
$b_{1},\ldots,b_{n} \in \overline{bAb}$ such that 
$a \precsim_{s} b_{i}$ for all $i=1,\ldots,n$.

\end{enumerate}
\end{nota}

\begin{rem}\label{rmkik}
Observe that  if $p$ is a projection in a C*-algebra $A$ and $a \in A_{+}$,
then $p \precsim_{s} a$ if and only if $p \precsim_{A} a$
(see Lemma~\ref{lemmvn}). 
\end{rem}

\begin{df}[see \cite{Lin01},  Definition~3.1]\label{deftr}
A unital C*-algebra $A$ is said to have \emph{tracial rank zero} 
if for any $\varepsilon >0$, any finite subset $F\subseteq A$ containing a nonzero element $b \geq 0$, 
any $\sigma_{1},\sigma_{2},\sigma_{3},\sigma_{4}$ with
$0<\sigma_{3} < \sigma_{4} < \sigma_{1} <\sigma_{2} < 1$, and any integer $n >0$, 
there exist a nonzero projection $p \in A$ and a  finite dimensional
C*-subalgebra $E \subseteq A$   with $1_{B}=p$, such that 
\begin{enumerate}
\item\label{deftr_it1}
 $\|pa-ap\| < \varepsilon$ for all $a \in F$;
\item\label{deftr_it2} 
$pFp \subseteq _{\varepsilon} B	$;	
\item\label{deftr_it3} 
$1-p \precsim_{s,n} p$ and 
$f_{\sigma_{1}}^{\sigma_{2}}((1-p)b(1-p)) \precsim_{s,n} f_{\sigma_{3}}^{\sigma_{4}}(pbp)$.
\end{enumerate}
If $A$ has tracial rank zero, we will write $\mathrm{TR}(A)=0$. 
A nonunital C*-algebra $A$ is said to have $\mathrm{TR}(A)=0$ if 
$\mathrm{TR}(A^{\sim})=0$.
\end{df}
Note that in \cite{Lin01}, for any nonnegative integer $k$, the notion
of a C*-algebra with tracial rank $k$ is introduced.
Lin also introduced a weaker version of the tracial rank zero  as follows.

\begin{df}[see \cite{Lin01}, Definition~3.4]\label{deftrw}
Let $A$ be a unital C*-algebra. We write $\mathrm{TR}_{w}(A)=0$ if for any $\varepsilon>0$, 
any finite subset $F\subseteq A$ containing a nonzero element $b \geq0$, any integer $n >0$, and any full element 
$x \in A_{+}$, there exist a nonzero projection $p \in A$ and a  
finite dimensional
C*-subalgebra $E \subseteq A$
with $1_{B}=p$, such that 
\begin{itemize}
\item[(1)]
 $\|pa-ap\| < \varepsilon$ for all $a \in F$;
\item[(2)]
$pFp \subseteq _{\varepsilon} B$ and $\|pbp\| \geq \|b\|-\varepsilon$;	
\item [(3)]
$ 1-p  \precsim_{s,n} p$ and  $1-p \precsim_{s} x$.
\end{itemize}
\end{df}

If $A$ is nonunital we write
$\mathrm{TR}_{w}(A)=0$ if $ \mathrm{TR}_{w}(A^{\sim})=0$.
 Observe that for any C*-algebra $A$,
$\mathrm{TR}_{w}(A)=0$ if and only if $A$ is TAF in the sense of \cite{Lin01c}.

\begin{rm}\label{rmkpos}
Note that in  Definition~\ref{deftrw}, we may omit the assumption that $b$ is positive.  
In fact, if $b$ is not positive we may assume
that $\|b\|=1$ and then use $b^{*}b$ instead of $b$.
\end{rm}

The following theorem follows from \cite[Theorem~6.13 and Remark~6.12]{Lin01}.

\begin{thm}\label{thmtr}
 Let $A$ be a simple  unital C*-algebra. Then  the following
statements are equivalent:
\begin{itemize}
\item[(a)]
$\mathrm{TR}(A)=0$;
\item[(b)]
$\mathrm{TR}_{w}(A)=0$;
\item[(c)]
for every finite set $F \subseteq A$, every $\varepsilon>0$, and every nonzero positive element
$x\in A$, there is  a nonzero C*-subalgebra
$B\subseteq A$ with $B\in\mathcal{I}^{(0)}$  such that, with $p=1_{B}$,  the following hold: 
\begin{itemize}
\item[(1)]
$\|pa-ap\|<\varepsilon$ for all $a\in F$;
\item[(2)]
$pFp\subseteq_{\varepsilon} B$;
\item[(3)]
$1-p\precsim_{A} x$.
\end{itemize}
\end{itemize}
Moreover, $\mathrm{TR}_{w}(A)= 0$ if and only if $\mathrm{TR}(A)= 0$.
\end{thm}

\noindent

The following is the main result
of the appendix.

\begin{thm}\label{main.trk}
Let $A$ be a simple   C*-algebra. 
Then $A$ has tracial rank zero if and only if
it has an approximate identity (not necessarily increasing)  
consisting of projections and
 for any finite set $F \subseteq A$, 
any $ x, y \in A_{+}$ with $x \neq 0$, and any $\varepsilon >0$, 
there exists a finite dimensional C*-sublagebra $E \subseteq A$  such that,
with  $p=1_{E}$, the following hold:
\begin{itemize}
\item[(1)] $\|pa-ap\| < \varepsilon$ for all $a \in F$;
\item[(2)] $pFp \subseteq _{\varepsilon} E$;
\item[(3)] $(y^{2} -ypy -\varepsilon)_{+} \precsim_{A} x$;
\item[(4)] $\|pxp\| > \|x\| -\varepsilon$.
\end{itemize}

\end{thm}

The proof of this theorem needs some preparation and will be presented
after Lemma~\ref{lemmatrixk}.
Note that Theorem~\ref{main.trk} unifies the definitions of tracial 
rank zero for simple unital and simple nonunital  cases. 

\subsection{C*-algebras with Property~$(\mathrm{T}_{0})$}
To  prove  Theorem~\ref{main.trk},  in this subsection we define  Property~$(\mathrm{T}_{0})$ and 
study  some of its properties.

\begin{df}\label{deftk}
Let $A$ be a simple  C*-algebra. 
We say that $A$ has \emph{Property}~$(\mathrm{T}_{0})$ 
 if $A$ has an approximate identity (not necessarily increasing) consisting of projections and
for all positive elements $x,y\in A$ with $x\neq 0$, every finite set $F \subseteq A$, and
every $\varepsilon>0$,  there is  a 
finite  dimensional C*-subalgebra
$E\subseteq A$  such that, with $p=1_{E}$,  the following hold: 
\begin{enumerate}
\item\label{deftk_it1}
$\|pa-ap\|<\varepsilon$ for all $a\in F$;
\item\label{deftk_it2}
$pFp\subseteq_{\varepsilon}E$;
\item\label{deftk_it3}
$(y^{2}-ypy-\varepsilon)_{+} \precsim_{A} x$;
\item\label{deftk_it4}
$\|pxp\|>\|x\|-\varepsilon$.
\end{enumerate}
\end{df}

We need the following lemma in the sequel.
The proof is very similar to that of  Lemma~\ref{lemdefwtrp}
and so it is omitted.

\begin{lma}\label{lemmkey2}
Let $A$ be a C*-algebra and let $x \in A_{+}\setminus \{0\}$. 
Suppose that $y \in A_{+}$ has the following property. 
For any finite set $F \subseteq A$ and  any $ \varepsilon > 0$ 
there exist a projection $p \in A$ and a
finite  dimensional C*-subalgebra $E \subseteq A$  
with unit $p$ such that the following hold:
\begin{enumerate}
\item\label{lemmkey2_it1}
 $\|pa-ap\| < \varepsilon$ for all $a \in F$;
\item\label{lemmkey2_it2}
$pFp\subseteq_{\varepsilon}E$;
\item\label{lemmkey2_it3}
$(y^{2}-ypy-\varepsilon)_{+} \precsim_{A} x$;
\item\label{lemmkey2_it4}
$\|pxp\|>\|x\|-\varepsilon$.
\end{enumerate}
Then every positive element $z\in\overline{Ay}$ also has the same property.
\end{lma}

The following proposition shows the relation between Property~$(\mathrm{T}_{0})$ and  
tracial rank zero for simple  unital C*-algebras.

\begin{prop}\label{propequstaf}
Let $A$ be a simple  unital C*-algebra.
The following statements are equivalent:
\begin{enumerate}
\item\label{propequstaf_it1} 
$A$ is has Property~$(\mathrm{T}_{0})$;
\item\label{propequstaf_it2}
 $\mathrm{TR}(A)=0$;
\item\label{propequstaf_it3}
 $\mathrm{TR}_{w}(A)=0$;
\item\label{propequstaf_it4}
 for any $x,y,\varepsilon,F$ as in Definition~\ref{deftk}
there is  a nonzero 
finite dimensional C*-subalgebra
$E\subseteq A$  such that \eqref{deftk_it1},
\eqref{deftk_it2}, and \eqref{deftk_it3}   in 
Definition~\ref{deftk} hold.
\end{enumerate}
\end{prop}

\begin{proof}
By  Theorem~\ref{thmtr} we have \eqref{propequstaf_it2}$\Leftrightarrow$\eqref{propequstaf_it3}.
The implication \eqref{propequstaf_it1}$ \Rightarrow $\eqref{propequstaf_it4} is obvious. 
Moreover, \eqref{propequstaf_it4}$ \Rightarrow $\eqref{propequstaf_it3} 
follows from Theorem~\ref{thmtr}  by applying \eqref{propequstaf_it4} with $y=1$ and using the fact that
for any positive  number $  \varepsilon <1$,
\begin{equation}\label{equstaf}
(1-p-\varepsilon)_{+}=(1-\varepsilon)(1-p)\sim_{A} (1-p).
\end{equation}
To see \eqref{propequstaf_it3}$ \Rightarrow $\eqref{propequstaf_it1}, 
note that \eqref{propequstaf_it3}  together with \eqref{equstaf} imply that
 Definition~\ref{deftk} is satisfied for $y=1$. Now by  Lemma~\ref{lemmkey2}, 
  Definition~\ref{deftk} is satisfied for every $y\in A_{+}$.
 Therefore, \eqref{propequstaf_it1}  holds.
\end{proof}

We need the following lemma in the proof of Proposition~\ref{propstafnu}.

\begin{lma}\label{lemheru}
Let $A$ be a simple C*-algebra with Property~$(\mathrm{T}_{0})$.
 Then every unital hereditary C*-subalgebra
of $A$ also has Property~$(\mathrm{T}_{0})$.
\end{lma}

\begin{proof}
Let $B=qAq$ be a unital hereditary C*-subalgebra of $A$ where $q$ is a projection of $A$.
Let $F\subseteq B$ be a finite subset, let $x,y\in B_{+}$ with $x\neq 0$, and 
let $\varepsilon>0$.
We may assume that $F\cup\{x,y\}$ is contained in the closed unit ball of $B$. 
Put $G=F\cup\{q\}$. Choose $\delta>0$ with $\delta<\min\{\frac{1}{6}, \frac{\varepsilon}{43}\}$.
Since
$A$ has Property~$(\mathrm{T}_{0})$,  there is a subalgebra $E\subseteq A$ in $\mathcal{I}^{(0)}$
such that, with $p=1_{E}$,  the following hold: 
\begin{enumerate}
\item\label{lemheru_it1}
$\|pa-ap\|<\delta$ for all $a\in F$;
\item\label{lemheru_it2}
$pGp\subseteq_{\delta}E$;
\item\label{lemheru_it3}
$(y^{2}-ypy-\delta)_{+} \precsim_{A} x$;
\item\label{lemheru_it4}
$\|pxp\|>\|x\|-\delta$.

\setcounter{TmpEnumi}{\value{enumi}}
\end{enumerate}
Then by \eqref{lemheru_it1} we have
\[
\|(qpq)^{2}-qpq\|=\|qpqpq-qppq\|\leq \|qpq-pq\|<\delta.
\]
Thus by \cite[Lemma~2.5.5]{Lin.book} 
(note that the assumption $\|a\|\geq \frac{1}{2}$ is 
unnecessary in the statement of that lemma),
there is a projection $q_{1}\in B$ such that:
\begin{enumerate}
\setcounter{enumi}{\value{TmpEnumi}}

\item\label{lemheru_it5}
$\|q_{1}-qpq\|<2\delta$.

\setcounter{TmpEnumi}{\value{enumi}}
\end{enumerate}
By \eqref{lemheru_it2} there is $c\in E$ such that $\|qpq-c\|<\delta$. Then $\|q_{1}-c\|<3\delta$. 
Thus by \cite[Lemma~2.5.4]{Lin.book} (note that the assumption that $a$ is self-adjoint is 
unnecessary in the statement of that lemma),
there is a projection $e\in E$ such that:
\begin{enumerate}
\setcounter{enumi}{\value{TmpEnumi}}

\item\label{lemheru_it6} 
$\|q_{1}-e\|<6\delta$.

\setcounter{TmpEnumi}{\value{enumi}}
\end{enumerate}
Hence by \cite[Lemma~2.5.1]{Lin.book}, there is a unitary $u\in A^{\sim}$ such that:
\begin{enumerate}
\setcounter{enumi}{\value{TmpEnumi}}

\item\label{lemheru_it7} 
$u^{*}eu=q_{1}$ and $\|u-1_{A^{\sim}}\|<12\delta$.

\setcounter{TmpEnumi}{\value{enumi}}
\end{enumerate}
Put $D=u^{*}eEeu$. Then $D$ is in $\mathcal{I}^{(0)}$ and 
$D=q_{1}u^{*}Euq_{1}\subseteq qAq=B$. Also,  $1_{D}=u^{*}eu=q_{1}$.
We show that:
\begin{enumerate}
\setcounter{enumi}{\value{TmpEnumi}}

\item\label{lemheru_it8}
 $\|q_{1}a-aq_{1}\| < \varepsilon$ for all $a \in F$;
 
\item\label{lemheru_it9}
$q_{1}Fq_{1} \subseteq _{\varepsilon} D$;	

\item\label{lemheru_it10}
$(y^{2}-yq_{1}y-\varepsilon)_{+} \precsim_{B} x$;

\item\label{lemheru_it11}
$\|q_{1}xq_{1}\|>\|x\|-\varepsilon$.

\setcounter{TmpEnumi}{\value{enumi}}
\end{enumerate}
By \eqref{lemheru_it1} and \eqref{lemheru_it5} we have 
$\|q_{1}-pq\|\leq \|q_{1}-qpq\|+\|qpq-pq\|<3\delta$. Thus,

\begin{enumerate}
\setcounter{enumi}{\value{TmpEnumi}}

\item\label{lemheru_it12}
 $\|q_{1}-pq\|=\|q_{1}-qp\|<3\delta$.

\end{enumerate}
To see \eqref{lemheru_it8}, by \eqref{lemheru_it1} and 
\eqref{lemheru_it12} for all $a\in F$ we have
\[
\|q_{1}a-aq_{1}\|\leq \|q_{1}a-pqa\|+\|pa-ap\|+\|aqp-aq_{1}\|<7\delta<\varepsilon.
\]
To show \eqref{lemheru_it9} let $a\in F$. By \eqref{lemheru_it2} 
there is $b\in E$ such that $\|pap-b\|<\delta$. Put
$d=u^{*}ebeu\in D$. Then by \eqref{lemheru_it6}, 
\eqref{lemheru_it7}, and \eqref{lemheru_it12} we have:
\begin{align*}
\|q_{1}aq_{1}-d\|&\leq \|q_{1}aq_{1}-eq_{1}aq_{1}e\|
+\|eq_{1}aq_{1}e-epqaqpe\| \\
&\ \ \ +\|epape-ebe\|+\|ebe-d\|\\
&<12\delta+6\delta+\delta+24\delta=43\delta<\varepsilon.
\end{align*}
To see \eqref{lemheru_it11}, by \eqref{lemheru_it12} 
at the first step and by \eqref{lemheru_it4} at the third step we get
\[
\|q_{1}xq_{1}\|\geq \|pqxqp\|-6\delta =\|pxp\|-6\delta>\|x\|-7\delta>\|x\|-\varepsilon.
\]
To prove \eqref{lemheru_it10}, first by \eqref{lemheru_it12} we have
\begin{align*}
\|(y^{2}-ypy-\delta)_{+}-(y^{2}-yq_{1}y)\|&\leq \delta+
\|(y^{2}-ypy)-(y^{2}-yq_{1}y)\|\\
&=\delta+\|ypqy-yq_{1}y\| <4\delta<\varepsilon.
\end{align*}
Therefore, by Lemma~\ref{lemkr}, we get
$(y^{2}-yq_{1}y-\varepsilon)_{+} \precsim_{A}(y^{2}-ypy-\delta)_{+} \precsim_{A} x$.
Since $B$ is hereditary in $A$, we obtain 
$(y^{2}-yq_{1}y-\varepsilon)_{+}   \precsim_{B} x$.
This completes the proof of \eqref{lemheru_it8}--\eqref{lemheru_it11},
showing that $B$ has Property~$(\mathrm{T}_{0})$.
\end{proof}

To compare Property~$(\mathrm{T}_{0})$ with   tracial rank zero 
for simple  not necessarily unital C*-algebras, we need the following result.

\begin{prop}\label{propstafnu}
Let $A$ be a simple  nonunital C*-algebra.
Then 
 $A$ has Property~$(\mathrm{T}_{0})$ if and only if the following holds. For every $\varepsilon>0$,
every $n\in \mathbb{N}$, every nonzero positive element $x\in A^{\sim}$,
 every finite subset $F\subseteq A^{\sim}$ which contains a nonzero positive element $x_{1}$,
 and every $\sigma_{i}$, $1\leq i\leq 4$, with
 $0<\sigma_{3}<\sigma_{4}<\sigma_{1}<\sigma_{2}<1$,
there exists a finite dimensional C*-subalgebra $E\subseteq A$ 
 such that, with $p=1_{E}$,
the following hold:
\begin{enumerate}
 
\item\label{propstafnu_it1} 
$\|pa-ap\|<\varepsilon$ for all $a\in F$;
\item\label{propstafnu_it2} 
$pFp\subseteq_{\varepsilon} E$ and $\|px_{1}p\|\geq \|x_{1}\|-\varepsilon$;
\item\label{propstafnu_it3} 
$1-p\precsim_{A} x$  and $ 1-p \precsim_{s,n} p$;
\item\label{propstafnu_it4} 
$ f_{\sigma_{1}}^{\sigma_{2}}((1-p)x_{1}(1-p))
\precsim_{s,n} f_{\sigma_{3}}^{\sigma_{4}}(px_{1}p)$.

\setcounter{TmpEnumi}{\value{enumi}}
\end{enumerate}
\end{prop}

\begin{proof}
To prove the forward implication let $A$ be a simple  nonuintal C*-algebra
with Property~$(\mathrm{T}_{0})$.
By  Definition~\ref{deftk},
there is a net $(p_{i})_{i\in I}$ of projections in $A$ which is a
(not necessarily increasing) approximate identity for $A$. 
For any $y\in A^{\sim}$ we have
\begin{equation}\label{equai}
\|y\|=\lim_{i\to\infty}\|p_{i}yp_{i}\|.
\end{equation}
In fact, write $y=\lambda+a$ where $\lambda\in \mathbb{C}$ and $a\in A$. Since $A$
is not unital we have 
$\|y\|=\sup\{\|yb\|\colon b\in A\ \text{with}\ \|b\|\leq 1\}$.
Let $\delta>0$.
Then there is $b\in A $ with $\|b\|\leq 1$ such that $\|yb\|>\|y\|-\delta$. Note that
$p_{i}yp_{i}b=\lambda p_{i}b+p_{i}ap_{i}b$ which tends to
$\lambda b+ab=yb$. Thus there is $j\in I$ such that $\|p_{i}yp_{i}b\|>\|y\|-\delta$
for all $i\geq j$. Then for every $i\geq j$ we have
\[
\|y\|-\delta<\|p_{i}yp_{i}b\|\leq \|p_{i}yp_{i}\|\leq \|y\|,
\]
and so \eqref{equai} holds.

Next, let $\varepsilon$, $n$, $F$, $x_{1}$, and $x$ be as in the statement. Write
$F=\{x_{1},\ldots,x_{m}\}$
and $x_{j}=\lambda_{j}+a_{j}$ where $\lambda_{j}\in\mathbb{C}$ and $a_{j}\in A$ for all
$1\leq j\leq m$. 
Choose $d_{3},d_{4}$ with 
$\sigma_{4}<d_{3}<d_{4}<\sigma_{1}$.
By \cite[Lemma~2.6]{Lin01}, there exists $\eta>0$ such that if $a,b\in A^{\sim}$ are
positive elements with $\|a\|,\|b\|\leq \|x_{1}\|$ and $\|a-b\|<\eta$
then 
$[f_{d_{3}}^{d_{4}}(a)]\leq [f_{\sigma_{3}}^{\sigma_{4}}(b)]$.
(Note that in \cite[Lemma~2.6]{Lin01} it is assumed that $\|a\|,\|b\|\leq 1$ but the proof of 
this lemma works for any upper bound $M>0$ instead of 1.)
Choose $\delta$ with $0<\delta<\min\{\frac{\varepsilon}{4},\frac{\eta}{3}\}$.
By the previous remark and that 
$(p_{i})_{i\in I}$  is an approximate identity for $A$, there is $i\in I$ such that, with
$a_{j}'=p_{i}a_{j}p_{i}$, the following hold:
\begin{enumerate}
\setcounter{enumi}{\value{TmpEnumi}}

\item\label{propstafnu_it5} 
$\|p_{i}x_{1}p_{i}\|>\|x_{1}\|-\delta$;
\item\label{propstafnu_it6} 
$\|a_{j}p_{i}-a_{j}\|<\delta$ and $\|p_{i}a_{j}-a_{j}\|<\delta$ for all $1\leq j\leq m$;
\item\label{propstafnu_it7} 
$\|a_{j}'-a_{j}\|<\delta$ for all $1\leq j\leq m$;
\item\label{propstafnu_it8} 
$p_{i}xp_{i}\neq 0$.

\setcounter{TmpEnumi}{\value{enumi}}
\end{enumerate}
Set $G=\{a_{j}'\mid 1\leq j\leq m\}$.
By  Lemma~\ref{lemheru}, $B=p_{i}Ap_{i}$ has  Property~$(\mathrm{T}_{0})$ 
and
hence $\mathrm{TR}(B)\leq k$ by Proposition~\ref{propequstaf}.
Then, by \cite[Theorem~5.6]{Lin01}, 
there is a  C*-subalgebra $D\subseteq B$ with 
$D\in\mathcal{I}^{(0)}$ such that, with $q=1_{D}$,
the following hold:
\begin{enumerate}
\setcounter{enumi}{\value{TmpEnumi}}

\item\label{propstafnu_it9} 
$\|qa_{j}'-a_{j}'q\|<\delta$ for all $1\leq  j\leq m$;
\item\label{propstafnu_it10} 
$qGq\subseteq_{\delta} D$ and $\|qb_{1}q\|\geq \|b_{1}\|-\delta$ where $b_{1}=p_{i}x_{1}p_{i}$;
\item\label{propstafnu_it11} 
$p_{i}-q\precsim_{B} p_{i}xp_{i}$, and $ p_{i}-q\precsim_{s,n} q$.
\item\label{propstafnu_it12} 
$ f_{\sigma_{1}}^{\sigma_{2}}((p_{i}-q)b_{1}(p_{i}-q)) 
\precsim_{s,n} 
 f_{d_{3}}^{d_{4}}(qb_{1}q) $.

\setcounter{TmpEnumi}{\value{enumi}}
\end{enumerate}
Put $E=\mathbb{C}(1-p_{i})+D$ and $p=1-p_{i}+q$ which is the unit of $E$
(here $1$ denotes the unit of $A^{\sim}$). Then $E\in\mathcal{I}^{(0)}$.
Now we show that \eqref{propstafnu_it1}--\eqref{propstafnu_it4} in the statement hold. 
To see \eqref{propstafnu_it1}, by \eqref{propstafnu_it7} and \eqref{propstafnu_it9}, 
for all $1\leq  j\leq m$ we have
\begin{align*}
\|px_{j}-x_{j}p\|=\|pa_{j}-a_{j}p\| &\leq \|pa_{j}-pa_{j}'\|+
\|pa_{j}'-a_{j}'p\|+\|a_{j}'p-a_{j}p\|\\
&<2\delta+\|qa_{j}'-a_{j}'q\|<3\delta<\varepsilon.
\end{align*}
To see \eqref{propstafnu_it2}, fix  $1\leq  j\leq m$. By 
\eqref{propstafnu_it10} there is $d\in D$ such that 
$\|qa_{j}'q-d\|<\delta$. Put $e=\lambda_{j}p+d\in E$. Then by \eqref{propstafnu_it6}
at the fifth step we have
\begin{align*}
\|px_{j}p-e\|&=\|pa_{j}p-d\|\leq \|qa_{j}'q-d\|+\|pa_{j}p-qa_{j}'q\|\\
&<\delta+\|(1-p_{i})a_{j}(1-p_{i})+(1-p_{i})a_{j}q+qa_{j}(1-p_{i})\|\\
&<\delta+2\|a_{j}-p_{i}a_{j}\|+\|a_{j}-a_{j}p_{i}\| <4\delta<\varepsilon.
\end{align*}
For the second part of \eqref{propstafnu_it2}, by 
\eqref{propstafnu_it6} at the third step, by \eqref{propstafnu_it10}
at  the fifth step, and by \eqref{propstafnu_it5} at the sixth step   we get

\begin{align*}
\|px_{1}p\|&\geq 
\|(1-p_{i})x_{1}(1-p_{i})+qx_{1}q\| 
 -
\|(1-p_{i})x_{1}q\|-\|qx_{1}(1-p_{i})\|\\
&=\max\left\{\|(1-p_{i})x_{1}(1-p_{i})\|,\, \|qx_{1}q\|\right\} 
 -\|(1-p_{i})a_{1}q\|-\|qa_{1}(1-p_{i})\|\\
&>\|qx_{1}q\|-2\delta=\|qp_{i}x_{1}p_{i}q\|-2\delta\\
&\geq \|p_{i}x_{1}p_{i}\|-3\delta >\|x_{1}\|-4\delta >\|x_{1}\|-\varepsilon.
\end{align*}

To prove \eqref{propstafnu_it3}, note that $1-p=p_{i}-q$. Thus by 
\eqref{propstafnu_it11},
$ 1-p \precsim_{s,n} q\precsim_{s} p$. Also, we have
$1-p=p_{i}-q\precsim_{A} p_{i}xp_{i}\precsim_{A} x$.

To see \eqref{propstafnu_it4}, first note that
\[
(p_{i}-q)b_{1}(p_{i}-q))=(p_{i}-q)x_{1}(p_{i}-q))=(1-p)x_{1}(1-p)).
\]
Thus by \eqref{propstafnu_it12}, 
$ f_{\sigma_{1}}^{\sigma_{2}}((1-p)x_{1}(1-p))\precsim_{s,n} 
 f_{d_{3}}^{d_{4}}(qb_{1}q)
 = f_{d_{3}}^{d_{4}}(qx_{1}q) $.
So to prove \eqref{propstafnu_it4} it is enough to show that 
\begin{equation}\label{equfd1}
 f_{d_{3}}^{d_{4}}(qx_{1}q) \precsim_{s} f_{\sigma_{3}}^{\sigma_{4}}(px_{1}p).
\end{equation}
For this, first by \eqref{propstafnu_it6} we have (recall that $x_{1}=\lambda_{1}+a_{1}$):
\begin{align*}
\|px_{1}p-(qx_{1}q &+\lambda_{1}(1-p_{i}))\|\\
&=\|(1-p_{i})x_{1}q+qx_{1}(1-p_{i}) 
 +
(1-p_{i})x_{1}(1-p_{i}) 
  -\lambda_{1}(1-p_{i})\| \\
&=
\|(1-p_{i})a_{1}q+qa_{1}(1-p_{i})
+(1-p_{i})a_{1}(1-p_{i})\|\\
& <3\delta<\eta.
\end{align*}
On the other hand, we have $\|px_{1}p\|\leq\|x_{1}\|$ and 
\[
\|qx_{1}q+\lambda_{1}(1-p_{i})\|=
\max\{\|qx_{1}q\|,\ \|\lambda_{1}(1-p_{i})\|\}\leq\|x_{1}\|.
\]
Thus, by the choice of $\eta$, we get
\begin{equation}\label{equfd2}
 f_{d_{3}}^{d_{4}}(qx_{1}q+\lambda_{1}(1-p_{i}))\precsim_{s}  f_{\sigma_{3}}^{\sigma_{4}}(px_{1}p).
\end{equation}
Also, since $qx_{1}q\perp \lambda_{1}(1-p_{i})$ and $f_{d_{3}}^{d_{4}}(0)=0$ we have
\[
f_{d_{3}}^{d_{4}}(qx_{1}q+\lambda_{1}(1-p_{i}))=
f_{d_{3}}^{d_{4}}(qx_{1}q)+
f_{d_{3}}^{d_{4}}(\lambda_{1}(1-p_{i})),
\]
and hence,
\begin{equation}\label{equfd3}
 f_{d_{3}}^{d_{4}}(qx_{1}q) \precsim_{s}
 f_{d_{3}}^{d_{4}}(qx_{1}q+\lambda_{1}(1-p_{i})) .
\end{equation}
Combining \eqref{equfd2} and \eqref{equfd3}, we get \eqref{equfd1},
and hence \eqref{propstafnu_it4} follows.

Now we prove the backward implication.
Suppose that the condition of the statement holds  (we do not use 
\eqref{propstafnu_it4} in the proof).
We show that $A$ has  Property~$(\mathrm{T}_{0})$.
Note that this condition is stronger than the definition of  
$\mathrm{TR}_{w}(A^{\sim})=0$ (that is, $A^{\sim}$ is TAF),
since it is not assumed that $x\in  (A^{\sim})_{+}$ is full.
Observe that in \cite[Proposition~2.7]{Lin01c} the assumption
that $a\in (A^{\sim})_{+}$ is full is not used in the proof of both parts.
Now let $\varepsilon>0$, let $x,y\in A_{+}$,  and let $F\subseteq A$ be as in 
Definition~\ref{deftk}.
By  Lemma~\ref{lemmkey2} we may assume that $\|y\|\leq 1$.
Then by (the proof of) \cite[Proposition~2.7]{Lin01c} with $x$ in place of $x_{1}$,
with $\mathcal{F}=F\cup\{x,y\}$, with $\frac{\varepsilon}{3}$ in place of $\varepsilon$,
and with $x$ in place of $a$, we obtain two orthogonal projections $p_{1},p_{2}\in A$
and a finite dimensional C*-subalgebra $E\subseteq A$
such that $p_{1}=1_{E}$ and that the following hold:
\begin{enumerate}
\setcounter{enumi}{\value{TmpEnumi}}

\item\label{propstafnu_it13}
$\|p_{i}a-ap_{i}\|<\frac{\varepsilon}{3}$ for all $a\in \mathcal{F}$ and $i=1,2$;
\item\label{propstafnu_it14}
$p_{1}\mathcal{F}p_{1}\subseteq_{\frac{\varepsilon}{3}}E$, 
$\|p_{1}xp_{1}\|\geq \|x\|-\frac{\varepsilon}{3}$, and
$\|(p_{1}+p_{2})a-a\|<\frac{\varepsilon}{3}$ for all $a\in \mathcal{F}$;
\item\label{propstafnu_it15}
$p_{2}\precsim_{A} x$.

\setcounter{TmpEnumi}{\value{enumi}}
\end{enumerate}
By \eqref{propstafnu_it13} and \eqref{propstafnu_it14}, 
the finite dimensional C*-subalgebra $E\subseteq A$ with unit $p_{1}$
satisfies \eqref{deftk_it1}, \eqref{deftk_it2}, and \eqref{deftk_it4} in  
Definition~\ref{deftk}. To see \eqref{deftk_it3}, first by 
\eqref{propstafnu_it13} and \eqref{propstafnu_it14} we get 
\begin{align*}
\|y^{2}-yp_{1}y-yp_{2}y\|&\leq \|y^{2}-y^{2}p_{1}-y^{2}p_{2}\|+\|y^{2}p_{1}-yp_{1}y\|+
\|y^{2}p_{2}-yp_{2}y\|\\
&\leq 
\|y\|(\|y-y(p_{1}+p_{2})\|+\|yp_{1}-p_{1}y\|+\|yp_{2}-p_{2}y\|)\\
&<\frac{\varepsilon}{3}+\frac{\varepsilon}{3}+\frac{\varepsilon}{3}=\varepsilon.
\end{align*}
Then by \eqref{propstafnu_it15} and Lemma~\ref{lemkr} we have 
$(y^{2}-yp_{1}y-\varepsilon)_{+}\precsim_{A} yp_{2}y\precsim_{A} p_{2}\precsim_{A} x$.
Therefore, $A$ has Property~$(\mathrm{T}_{0})$, as desired. (Note that \eqref{propstafnu_it14}
also implies that
$A$ has an approximate identity consisting of projections.)
\end{proof}

Observe that  Proposition~\ref{propstafnu} holds also for any simple  \emph{unital} C*-algebra $A$
(note that in this case $A^{\sim}=A$ according to our convention).
This follows from  Proposition~\ref{propequstaf} and 
\cite[Theorem~5.6]{Lin01}.

\begin{rem}\label{rmkt0tr}
Let $A$ be a simple  nonunital C*-algebra. Then
$A$ has Property~$(\mathrm{T}_{0})$ (equivalently, $\mathrm{TR}(A)=0$ by
 Theorem~\ref{main.trk}) if and only if 
 Conditions~\eqref{propstafnu_it1}--\eqref{propstafnu_it3} in  Proposition~\ref{propstafnu}
hold (because Condition~\eqref{propstafnu_it4} is not 
used in the proof of the converse
of  Proposition~\ref{propstafnu}).
Thus
the only difference between the notion of having Property~$(\mathrm{T}_{0})$
(equivalently, $\mathrm{TR}(A)=0$) and
$\mathrm{TR}_{w}(A)=0$  is that  in the definition of $\mathrm{TR}_{w}(A)=0$
(Definition~\ref{deftrw})
it is required  that the nonzero positive
element $x\in A^{\sim}$ is full.
\end{rem}


Now, we can prove one direction of  Theorem~\ref{main.trk}.
 
\begin{prop}\label{propstaf}
Let $A$ be a simple 
 C*-algebra with Property~$(\mathrm{T}_{0})$.
Then $\mathrm{TR}(A)=0$.
\end{prop}

\begin{proof}
If $A$ is a simple  unital C*-algebra then $A$ has Property~$(\mathrm{T}_{0})$ if and only if  
$\mathrm{TR}(A)=0$, by  Proposition~\ref{propequstaf}.
Let $A$ be a simple  nonunital  C*-algebra with Property~$(\mathrm{T}_{0})$. 
Then  Proposition~\ref{propstafnu} and
 Definition~\ref{deftr} imply that $\mathrm{TR}(A)=0$, as desired.
\end{proof}

\subsection{Permanence properties}
In this subsection we study some permanence properties
of Property~$(\mathrm{T}_{0})$, and we give
the proof of Theorem~\ref{main.trk}.
We begin with the following proposition which shows
that if a simple
C*-algebra has the local Property~$(\mathrm{T}_{0})$ then it has Property~$(\mathrm{T}_{0})$.

\begin{prop}\label{proploc}
Let $A$ be a simple C*-algebra with the following property:
for every $ \varepsilon > 0 $ 
	and every  finite subset $ F \subseteq A$ there exists a simple C*-subalgebra $ B$ 
	of $A$ with Property~$(\mathrm{T}_{0})$	such that  $ F \subseteq_{\varepsilon}B $.
Then $A$ has Property~$(\mathrm{T}_{0})$.
\end{prop}

\begin{proof}
Let $A$ be a simple C*-algebra with  the property in the statement.
Observe that $A$ has an approximate identity
(not necessarily increasing) consisting of projections.
Let $x$, $y$, $\varepsilon$, and  $F$ be as in Definition~\ref{deftk}. 
We may assume that $\varepsilon<1$ and $\|x\|=1$. 
Also, by  Lemma~\ref{lemmkey2} we may assume that
$\|y\|<\frac{1}{2}$.
 Write $F = \{f_{1}, \ldots, f_{m}\}$.
Choose $ \delta >0$ such that $\delta <\frac{\varepsilon}{4}$ and 
$ (2 + \delta)\delta < \frac{\varepsilon}{12}$.
Set $\tilde{F} = F \cup \{x^{\frac{1}{2}}, y^{\frac{1}{2}}\}$.
By assumption there is a  simple C*-subalgebra $B$ of $A$ with Property~$(\mathrm{T}_{0})$ such that
	$\tilde{F} \subseteq_{\delta}B $.
	Thus there is $ b \in B $ such that $\|x^{\frac{1}{2}} - b\|< \delta $. Then
	\begin{align}\label{proploc_eq1}
	 \|b^{*}b - x\| & \leq
	\| b^{*}b - b^{*}x^{\frac{1}{2}}\|
	+ \|b^{*}x^{\frac{1}{2}} - x \| \\ \notag
	&  \leq \|b\|\delta + \|x^{\frac{1}{2}}\|\delta \\
	& \leq (\|x^{\frac{1}{2}}\| + \delta + \|x^{\frac{1}{2}}\|)\delta
	< \tfrac{\varepsilon}{12}  \notag.
	\end{align}
Also, there exists $c \in B $ such that $ \| y^{\frac{1}{2}} - c\| < \delta$. Similarly, 
we have $\|c^{*}c - y\| < \tfrac{\varepsilon}{12}$.
Set $w = c^{*}c $. So $\|w\| < 1$. Also set $d=(b^{*}b-\frac{\varepsilon}{12})_{+}$.
Note that $d\neq 0$ since $\|b^{*}b\|>1-\frac{\varepsilon}{12}>\frac{1}{2}$
and $\frac{\varepsilon}{12}<\tfrac{1}{2}$.
Moreover, there exist $b_{1}, \ldots, b_{m} \in B $ such that $ \| b_{i} - f_{i }\| < \delta$ 
for all $ i = 1, \ldots, m $. 
Put $D= \{b_{1}, \ldots, b_{m}\}$.
Since $B$ has Property~$(\mathrm{T}_{0})$, by  Definition~\ref{deftk} 
there is a subalgebra
$E \subseteq B$ in $\mathcal{I}^{(0)}$ such that, with $p=1_{E}$,  the following hold: 
\begin{enumerate}

\item\label{proploc_it1}
$\|pb_{i}-b_{i}p\|<\delta$ for all $i=1,\ldots,m$;
\item\label{proploc_it2}
$pDp\subseteq_{\delta}E$;
\item\label{proploc_it3}
$(w^{2}-wpw-\delta)_{+} \precsim_{B} d$;
\item\label{proploc_it4}
$\|pdp\|>\|d\|-\delta$.
\end{enumerate}
Now we verify Conditions~\eqref{deftk_it1}--\eqref{deftk_it4} in  Definition~\ref{deftk} 
for the given $F ,x, y,\varepsilon $. For  Condition~\eqref{deftk_it1}, using
\eqref{proploc_it1} above and 
$ \| b_{i} - f_{i }\| < \delta$ we have
\[
\|pf_{i}-f_{i}p\|\leq \|pf_{i}-pb_{i}\|+\|pb_{i}-b_{i}p\|+\|b_{i}p-f_{i}p\|<3\delta<\varepsilon.
\]
To see Condition~\eqref{deftk_it2}, 
 fix $1\leq i\leq m$. By \eqref{proploc_it2} above there is $e\in E$ such that $\|pb_{i}p-e\|<\delta$. Then we have
\[
\|pf_{i}p-e\|\leq \|pf_{i}p-pb_{i}p\|+\|pb_{i}p-e\|<2\delta<\varepsilon.
\]
To show Condition~\eqref{deftk_it4} in  Definition~\ref{deftk}, using
\eqref{proploc_it4} at the second step and 
\eqref{proploc_eq1} at the fourth step   
  we get
\begin{align*}
\|pxp\|&\geq \|pdp\|-\|p(x-d)p\|\\
&>\|d\|-\delta -\tfrac{\varepsilon}{12}\\
&>\|b^{*}b\|-\tfrac{\varepsilon}{12}-
\tfrac{\varepsilon}{4}-\tfrac{\varepsilon}{12}\\
&>\|x\|-\tfrac{\varepsilon}{12}-\tfrac{\varepsilon}{2}>\|x\|-\varepsilon.
\end{align*}
To prove  Condition~\eqref{deftk_it3} in  Definition~\ref{deftk}, first    
using the inequalities $\|w - y\| < \tfrac{\varepsilon}{12}$ and
 $\|w\| < 1$, we obtain
\begin{align*}
	\|(w^{2} - wpw - \delta)_{+} &- (y^{2} - ypy)\| \\
	& \leq
	\| (w^{2} - wpw - \delta)_{+} - (w^{2} - wpw)\| \\
	& + \|(w^{2} - wpw) - (y^{2} - ypy) \| \\
	&\leq
	\delta + \|w^{2} - y^{2}\| + \|wpw - wpy \| 
	 + \|wpy  - ypy \| \\
	 &
	\leq \tfrac{\varepsilon}{4} + \tfrac{\varepsilon}{6} + \tfrac{\varepsilon}{12} + \tfrac{\varepsilon}{12} <\varepsilon.
	\end{align*}
	Therefore, by \eqref{proploc_it3} and Lemma~\ref{lemkr},
	$
	(y^{2} - ypy - \varepsilon)_{+} \precsim_{A} (w^{2} - wpw - 
	\delta)_{+} \precsim_{A} d\precsim_{A} x
	$.
	 This finishes the proof.
\end{proof}

The preceding proposition implies that the class of simple C*-algebras with
Property~$(\mathrm{T}_{0})$ is closed under taking arbitrary inductive limits.

The following characterization of Property~$(\mathrm{T}_{0})$ is essential in the following.

\begin{prop}\label{propappu}
Let $A$ be a simple C*-algebra.
Then $A$ has Property~$(\mathrm{T}_{0})$ if and only if there exists an approximate identity
(not necessarily increasing)
consisting  of projections 
$(p_{i})_{i \in I}$ for $A$ such that  
$\mathrm{TR}(p_{i}Ap_{i})=0$ for all $i\in I$.
\end{prop}

\begin{proof}
The forward implication follows from  Definition~\ref{deftk},  Lemma~\ref{lemheru},
and  Proposition~\ref{propstaf}. For the backward implication, let $(p_{i})_{i \in I}$ be as
in the statement.  Then Proposition~\ref{propequstaf} implies that each
$p_{i}Ap_{i}$ is a simple C*-algebra with Property~$(\mathrm{T}_{0})$. 
Let  $F\subseteq A$ be a finite subset and let $\varepsilon>0$. Since 
$(p_{i})_{i \in I}$ is an approximate identity, there exists $i\in I$ such that
$F\subseteq_{\varepsilon} p_{i}Ap_{i}$. 
Applying Proposition~\ref{proploc}, we conclude  that $A$ has Property~$(\mathrm{T}_{0})$.
\end{proof}

With the preceding characterization of Property~$(\mathrm{T}_{0})$, 
we can obtain more properties of simple C*-algebras with Property~$(\mathrm{T}_{0})$.

\begin{thm}\label{thmrrsr}
Let $A$ be a simple C*-algebra with Property~$(\mathrm{T}_{0})$. 
Then $A$ has real rank zero and stable rank one.
\end{thm}

\begin{proof}
Let $A$ be a nonzero simple C*-algebra with Property~$(\mathrm{T}_{0})$.
 Proposition~\ref{propappu} implies the existence of a nonzero projection $p \in A$ such that 
$\mathrm{TR}(pAp)=0$. 
Thus, by \cite[Theorem~7.1]{Lin01}, $pAp$ has real rank zero. 
Since $A$ is simple, $pAp$ is a full corner of $A$ and so $pAp$ is Morita equivalent to $A$. 
Then by \cite[Theorem~3.8]{BP91}, $A$ has also real rank zero. 
To see that $A$ has stable rank one, first note that \cite[Theorem~6.9]{Lin01} and 
\cite[Theorem~6.13]{Lin01} imply that $\mathrm{tsr}(pAp)=1$. 
Moreover, by \cite[Corollary~4.6]{Black}, $\mathrm{tsr}(A) \leq \mathrm{tsr}(pAp)$. 
Hence, $A$ has stable rank one.
\end{proof}

\begin{prop}\label{propher}
Let $A$ be a simple C*-algebra with Property~$(\mathrm{T}_{0})$ and let $B$ be a hereditary
C*-subalgebra of $A$. Then $B$ has Property~$(\mathrm{T}_{0})$ if and only if it has
an approximate identity (not necessarily increasing) consisting of 
 projections.
\end{prop}

\begin{proof}
The forward implication follows from Definition~\ref{deftk}.
For the backward implication
let $A$ be a simple C*-algebra with Property~$(\mathrm{T}_{0})$ and let $B$ be 
a hereditary C*-subalgebra of $A$ which contains an approximate identity consisting of projections 
$(p_{i})_{i \in I}$. For each $i\in I$ we have $p_{i}Bp_{i}=p_{i}Ap_{i}$
which has Property~$(\mathrm{T}_{0})$ by Lemma~\ref{lemheru}. Therefore,
Proposition~\ref{propequstaf} and Proposition~\ref{propappu} imply that $B$ has Property~$(\mathrm{T}_{0})$.
\end{proof}

\begin{cor}\label{thmstaf}
Let $A$ be a simple C*-algebra. The following are equivalent:
\begin{enumerate}

\item\label{thmstaf_it1}
$A$ has  Property~$(\mathrm{T}_{0})$;
\item\label{thmstaf_it2}
$\overline{xAx}$ has Property~$(\mathrm{T}_{0})$  for all $x\in A_{+}$;
\item\label{thmstaf_it3}
$A$ has real rank zero and $\mathrm{TR}(pAp)=0$ for all projections $p\in A$.

\end{enumerate}
\end{cor}

\begin{proof}
\eqref{thmstaf_it1}$\Rightarrow$\eqref{thmstaf_it2}: This follows from
Proposition~\ref{propher}
and the fact that $\overline{xAx}$ has real rank zero
(by Theorem~\ref{thmrrsr}).
\eqref{thmstaf_it2}$\Rightarrow$\eqref{thmstaf_it3}: 
Suppose that \eqref{thmstaf_it2} holds.
Then by Theorem~\ref{thmrrsr}, $\overline{xAx}$ has 
real rank zero,  for all $x\in A_{+}$.
It follows that $A$ has real rank zero. The second part
of \eqref{thmstaf_it3}  follows 
from Proposition~\ref{propstaf}.  
Finally, the implication 
\eqref{thmstaf_it3}$\Rightarrow$\eqref{thmstaf_it1}
follows from Proposition~\ref{propappu}.
\end{proof}

\begin{lma}\label{lemmatrixk}
Let $A$ be a simple  C*-algebra with Property~$(\mathrm{T}_{0})$.
Then $M_{n}(A)$ has Property~$(\mathrm{T}_{0})$ for all $n\in \mathbb{N}$.
\end{lma}

\begin{proof}
By Definition~\ref{deftk}, $A$ has 
an approximate identity  $(p_{i})_{i \in I}$  consisting of projections. 
Put $q_{i}=\mathrm{diag}(p_{i},\ldots,p_{i})\in M_{n}(A)$. 
 Then $(q_{i})_{i \in I}$ is
an approximate identity consisting  of projections 
 for $M_{n}(A)$.  Lemma~\ref{lemheru} and Proposition~\ref{propequstaf} yield that
 $\mathrm{TR}(p_{i}Ap_{i})=0$.
 Thus $q_{i}M_{n}(A)q_{i}=M_{n}(p_{i}Ap_{i})$ has tracial  rank zero,
 by \cite[Theorem~5.8]{Lin01}. Hence, Proposition~\ref{propappu} implies that  
 $M_{n}(A)$ has Property~$(\mathrm{T}_{0})$.
\end{proof}






Now, we are in a position to prove Theorem~\ref{main.trk}.

\begin{proof}[Proof of Theorem~\ref{main.trk}]
 The backward implication  follows from 
Proposition~\ref{propstaf}. 
For the other direction, let $A$ be a simple C*-algebra with $\mathrm{TR}(A)=0$.
We may assume that $A$ is nonunital since the unital case follows from
Proposition~\ref{propequstaf}.
As $\mathrm{TR}(A)=0$, \cite[Corollary~5.7]{Lin01} implies that
$\mathrm{TR}_{w}(A)=0$ (recall that, by definition, $\mathrm{TR}(A)=\mathrm{TR}(A^{\sim})$ and 
$\mathrm{TR}_{w}(A)=\mathrm{TR}_{w}(A^{\sim})$).
Thus, $A$ is TAF in the sense of \cite{Lin01c}.
Now by \cite[Corollary~2.8]{Lin01c}, $A$ has an 
approximate identity $(p_{i})_{i \in I}$ (not necessarily increasing)  consisting of projections.
(Note that the separability assumption is unnecessary in \cite[Corollary~2.8]{Lin01c}.)
Then
 it follows from \cite[Theorem~5.3]{Lin01} that $\mathrm{TR}(p_{i}Ap_{i})=0$
for all $i\in I$.
 Hence, Proposition~\ref{propappu} implies that $A$ has Property~$(\mathrm{T}_{0})$.
\end{proof}
 
 \begin{rem}\label{rmk_tr0tk}
 In view of Definition~\ref{deftk},
  Theorem~\ref{main.trk}
 says that a simple   C*-algebra $A$ has tracial
 rank zero if and only if it has  Property~$(\mathrm{T}_{0})$.
 \end{rem}
Theorem~\ref{main.trk} enables us to prove some permanence properties of
 simple C*-algebras of tracial rank zero which are not necessarily $\sigma$-unital.

 
  \begin{cor}[compare with \cite{Lin01}, Proposition~4.8]\label{corlimt0}
 Let $A$ be a simple C*-algebra which is an inductive limit of simple C*-algebras of 
 tracial   rank zero. Then $A$ has tracial rank zero.
 \end{cor}  
 
 \begin{proof}
 This follows from Theorem~\ref{main.trk}  and the remark after
 Proposition~\ref{proploc}.
  \end{proof}
  
  The following corollary was proved by Lin in the unital case.
  More precisely, Part~\eqref{corrrsr_it1} in the simple unital case follows from \cite[Theorem~7.1]{Lin01}
  and \cite[Theorem~3.6.11]{Lin.book}. Part~\eqref{corrrsr_it2} is proved in
  \cite[Theorem~5.8]{Lin01} in the unital not necessarily simple  case.
  Part~\eqref{corrrsr_it3} in the case of a unital hereditary subalgebra follows from 
  \cite[Theorem~5.3]{Lin01}. We deal with the nonunital case.
  
   \begin{cor}\label{corrrsr}
 Let $A$ be a simple   C*-algebra with tracial rank zero.
 Then the following hold:
\begin{enumerate}
\item\label{corrrsr_it1}
$A$  has real rank zero and stable rank one;
\item\label{corrrsr_it2}
$\mathrm{TR}(M_{n}(A))=0$ for all $n\in \mathbb{N}$;
\item\label{corrrsr_it3}
if $B$ is a hereditary
C*-subalgebra of $A$ then  $\mathrm{TR}(B)=0$.
\end{enumerate} 
 \end{cor}
 
 \begin{proof}
Part \eqref{corrrsr_it1} follows from Theorems~\ref{thmrrsr} and 
\ref{main.trk}. Also, 
Part~\eqref{corrrsr_it2} follows from Lemma~\ref{lemmatrixk} and Theorem~\ref{main.trk}.
Finally,
Part~\eqref{corrrsr_it3} follows from
Part~\eqref{corrrsr_it1}, Proposition~\ref{propher}, and Theorem~\ref{main.trk}.
  \end{proof}

\subsection{Morita equivalence}
 In  this subsection, we   prove that the class of simple C*-algebras with tracial rank zero 
 is  closed under Morita equivalence (note that we do not assume
 any separability condition).
 This result was used in the proof of Theorem~\ref{thmcross}.

\begin{prop}\label{propkk}
Let $A$ be a simple C*-algebra. Then 
 $A$ has Property~$(\mathrm{T}_{0})$ if and only if 
$A\otimes \mathcal{K}$ has Property~$(\mathrm{T}_{0})$. 
\end{prop}

\begin{proof}
The forward implication follows from  Lemma~\ref{lemmatrixk}, the remark
after Proposition~\ref{proploc},
and the fact that $A\otimes \mathcal{K}$ is isomorphic to an inductive limit 
$\varinjlim M_{n}(A)$.
The backward implication follows from
Corollary~\ref{corrrsr}\eqref{corrrsr_it3} and Theorem~\ref{main.trk}.
\end{proof}
 
 \begin{thm}\label{cormorita}
 Let $A$ be a nonzero simple   C*-algebra. 
 The following statement are equivalent:
 \begin{enumerate}
 \item\label{cormorita_it1}
 $\mathrm{TR}(A)=0$;
 \item\label{cormorita_it2}
 $A$ is Morita equivalent to a simple unital C*-algebra $B$ with
$\mathrm{TR}(B)=0$;
 \item\label{cormorita_it3}
 $\mathrm{TR}(pAp)=0$
for some (any) nonzero projection $p\in A$.
 \end{enumerate}
 In particular, the class of simple
  C*algebras with tracial rank zero
 is closed under Morita equivalence. 
 \end{thm}
 
 \begin{proof}
The implication \eqref{cormorita_it1}$\Rightarrow$\eqref{cormorita_it3} 
follows from Part~\eqref{corrrsr_it3} of Corollary~\ref{corrrsr}. Also,  
\eqref{cormorita_it3}$\Rightarrow$\eqref{cormorita_it2} 
is obvious.
For \eqref{cormorita_it2}$\Rightarrow$\eqref{cormorita_it1}, 
let $B$ be a simple  {unital}   C*-algebra 
with $\mathrm{TR}(B)=0$ such that
$B$ is Morita equivalent to $A$. 
By Part~\eqref{corrrsr_it1} of Corollary~\ref{corrrsr} we have
$\mathrm{RR}(B)=0$.
 By
\cite[Theorem~3.8]{BP91}, having real rank zero is preserved under Morita equivalence,
hence we get $\mathrm{RR}(A)=0$. 
 In particular,
$A$ has
an approximate identity (not necessarily increasing) consisting of projections 
$(p_{i})_{i \in I}$.
For each $i \in I$, the simple unital C*-algebra $p_{i}Ap_{i}$ is Morita equivalent 
to $A$, and so it is Morita equivalent to $B$.
Since both  $p_{i}Ap_{i}$ and $B$ are unital, they are stably isomorphic (by \cite{BGR}).
Thus by Proposition~\ref{propkk},  $p_{i}Ap_{i}$ has Property~$(\mathrm{T}_{0})$.
Hence, Propositions~\ref{propequstaf} and \ref{propappu} imply
 that $A$ has Property~$(\mathrm{T}_{0})$.
 Now,  
Theorem~\ref{main.trk} yields that $\mathrm{TR}(A)=0$.
  
 The equivalence of Parts~\eqref{cormorita_it1}
and \eqref{cormorita_it2}  implies that 
the class of simple
  C*algebras with tracial rank zero
 is closed under Morita equivalence. 
 \end{proof}
 
 As an application of the preceding
theorem,
 we give the following result. 

\begin{cor}\label{cor_compact}
Let $\alpha \colon G \rightarrow \mathrm{Aut}(A)$ 
be an action of a second countable 
 compact group 
$G$ on a simple separable unital
C*-algebra $A$ with tracial rank zero. 
Suppose that $\alpha$ has the Rokhlin property
in the sense of \cite{Gar14}. Then the 
crossed product 
$A \rtimes _{\alpha} G$ is a simple  
C*-algebra with tracial rank zero.
\end{cor}

 The proof is mainly based on  
 \cite[Theorem~4.5]{Gar14}  in which a similar
 result is obtained for the fixed point algebra.
 Note that the fixed point algebra is  unital,
 and so the original definition
 of tracial rank for unital C*-algebras can be applied. However, when $G$ is infinite,
 the crossed product is never unital.
 The Morita invariance of tracial rank zero
 for simple C*-algebras enables us to deal with 
 this difficulty.

\begin{proof}[Proof of Corollary~\ref{cor_compact}]
By  \cite[Theorem~4.5]{Gar14}, the fixed point
algebra $A^{\alpha}$ is a simple C*-algebra
with tracial rank zero. Also, by 
 \cite[Proposition~2.7]{Gar14}, the 
 fixed point algebra and the crossed product
 are Morita equivalent. Thus,
 Theorem~\ref{cormorita} implies that
 the 
crossed product 
$A \rtimes _{\alpha} G$ is also simple  
with tracial rank zero.
\end{proof}



\section*{Acknowledgments}
The authors would like to thank  N. Christopher Phillips for helpful discussions and valuable suggestions 
during the first author's visit at the CRM and Siegfried Echterhoff for helpful comments about Morita equivalence.
The second author would like to thank Huaxin Lin for some discussions via email.
The authors are grateful to Eusebio Gardella for reading the first draft of the paper
and giving useful comments.
The authors would like to thank the anonymous referee for valuable comments.
Part of this work was done during a visit of the first author at the University of M\"{u}nster
and her visit was supported by the Deutsche Forschungsgemeinschaft (SFB 878). The first author was supported by a grant from IPM.
The second author was supported by a grant from INSF (no.~98009270).

\bibliographystyle{amsplain}

\begin{thebibliography}{99}

\bibitem{AGJP17}
{\sc M. Amini, N. Golestani, S. Jamali, N. C. Phillips}, 
\emph{Simple tracially $\mathcal{Z}$-absorbing C*-algebras},
preprint.

\bibitem{APT11}
{\sc P.~Ara, F.~Perera,  A.~S.\  Toms},
{\emph{K-theory for operator algebras. Classification of C*-algebras}},
pages 1--71 in:
{\emph{Aspects of Operator Algebras and Applications}},
P.~Ara, F~Lled\'{o}, and F.~Perera (eds.),
Contemporary Mathematics vol.~534,
Amer.\  Math.\  Soc., Providence RI, 2011.

\bibitem{Archey}{\sc D. Archey},
{\it Crossed product C*-algebras by finite group actions with the tracial Rokhlin property,}
 Rocky Mountain J. Math. {\bf 41} (2011), no.~6, 1755--1768.

\bibitem{ABP17}{\sc D. Archey, J. Buck, N. C. Phillips},
{\it Centrally large subalgberas and tracial $\mathcal{Z}$-absorption,}
 Int. Math. Res. Not. IMRN 2018, no. 6, 1857--1877. 



 
\bibitem{Bla90}
{\sc B. Blackadar}, \emph{Symmetries of  the CAR algebras}, 
 Ann. of Math. (2) \textbf{131} (1990), no. 3, 589--623.
 
 
\bibitem{Black}
{\sc B. Blackadar}, \emph{The stable rank of full corners in C*-algebras},  
 Proc.  Amer. Math. Soc. \textbf{132} (2004), no. 10, 2945--2950.
 
 \bibitem{Bl06}
{\sc B. Blackadar}, \emph{Operator algebras: theory of C*-algebras and von Neumann algebras}, 
 Encyclopaedia of Mathematical Sciences, Vol.~122, 
 Springer-Verlag, Berlin, 2006.
 
\bibitem{BSKR93}
{\sc  O. Bratteli, E. St{\o}rmer, A. Kishimoto, M. R{\o}rdam},
\emph{The crossed product of a UHF~algebra by a shift}, Ergodic Theory Dynam. Systems 
\textbf{13} (1993), no.~4, 615--626. 
 
\bibitem{Br77}
{\sc L. G. Brown}, \emph{Stable isomorphism of hereditary subalgebras of C*-algebras}, 
Pacific J. Math.  \textbf{71}  (1977), no. 2, 335--348.

\bibitem{BGR}
{\sc L. G. Brown, P. Green, M. Rieffel}, \emph{Stable isomorphism and strong Morita equivalence  of C*-algebras}, 
Pacific J. Math.  \textbf{71}  (1977), no. 2, 349--368.

\bibitem{BP91}
{\sc  L. G. Brown, G. K. Pedersen}, \emph{C*-algebras of real rank zero},
J.~ Funct. Anal. \textbf{99} (1991), no.~1,  131--149.
 

\bibitem{DE02}
{\sc M.~Dadarlat, S.~Eilers}, \emph{On the
classification of nuclear C*-algebras}, Proc. London Math. Soc. \textbf{85} (2002), no.~3,
 168--210.
 
\bibitem{Echterhoff-Phillips}
{\sc  S. Echterhoff,  W. L\"{u}ck, N. C. Phillips, S. Walters},
{\it The structure of crossed products of irrational rotation algebras by finite subgroups of  $SL_{2}(\mathbb{Z})$,}
 J. Reine Angew. Math. \textbf{639} (2010), 173--221.  




\bibitem{Gar14} 
{\sc E. Gardella}, 
\emph{Crossed product of compact group actions with the Rokhlin property},
J. Noncommut. Geom. {\bf 11}(2017) 1593-1626.

\bibitem{GH20} 
{\sc E. Gardella,  I. Hirshberg}, 
\emph{Strongly outer actions of amenable groups on 
$\mathcal{Z}$-stable C*-algebras}, preprint 2020
(arXiv: 1811.00447v3 [math.OA]).

\bibitem{GHS} 
{\sc E. Gardella,  I. Hirshberg,  L. Santiago}, 
\emph{Rokhlin dimension: duality, tracial properties and  crossed products}, 
Ergod. Theory Dyn. Syst.,
Published online by Cambridge University Press: 18 October 2019. 

\bibitem{GS16} 
{\sc E. Gardella, L. Santiago}, 
\emph{Equivariant $*$-homomorphisms, Rokhlin constraints and equivariant UHF-absorption}, 
J. Funct. Anal. \textbf{270} (2016), no. 7, 2543--2590.

\bibitem{Ph17}
{\sc T. Giordano, D. Kerr, N. C. Phillips, A. S. Toms}, \emph{Crossed products of C*-algebras, 
topological dynamics, and classification},
  Advanced Courses in Mathematics. 
 CRM Barcelona. Birkh\"{a}user/Springer, Cham, 2018.


\bibitem{G19}
{\sc N. Golestani},   in preparation.




  
  
 \bibitem{HO84}
 {\sc R. Herman, A. Ocneanu,} {\it Stability for integer actions on UHF~C*-algebras,} 
  J. Funct. Anal. {\bf 59} (1984), no.~1, 132--144.


 \bibitem{Hirshberg}
 {\sc I. Hirshberg, J. Orovitz,} {\it Tracially $\mathcal{Z}$-absorbing C*-algebras,} 
  J. Funct. Anal. {\bf 265} (2013), no.~5,  765--785.
  

\bibitem{HRW07}
{\sc A. A. Huef, I. Raeburn, D. P. Williams}, 
{\it Properties preserved under Morita equivalence of C*-algebras}, 
Proc.  Amer. Math. Soc. \textbf{135} (2007),  no.~5, 1495--1503.

  \bibitem{IZ04a}
 {\sc M. Izumi,} {\it  Finite group actions on C*-algebras with the Rohlin property, I,} 
   Duke. Math. J. {\bf 122} (2004), no.~2,  233--280.
   
 \bibitem{IZ04b}
 {\sc M. Izumi,} {\it  Finite group actions on C*-algebras with the Rohlin property, II,} 
   Adv. Math. {\bf 184} (2004), no.~1,  119--160.

 \bibitem{Jac13}
 {\sc J. Jacelon}, \emph{A simple, monotracial, stably projectionless C*-algebra},
  J. Lond. Math. Soc. (2) \textbf{87} (2013), no.~2, 365--383.

\bibitem{Je95} 
 {\sc J. A. Jeong}, \emph{Purely infinite simple C*-crossed products},
    Proc.  Amer. Math. Soc. \textbf{123} (1995), no. 10, 3075--3078.
    
\bibitem{JO98} 
 {\sc J. A. Jeong, H. Osaka}, \emph{Extremely rich C*-crossed products and
the cancellation property}, J. Austral. Math. Soc. (Series A) \textbf{64} (1998), 285--301.
   
   
 
\bibitem{KR00} 
{\sc E. Kirchberg, M. R{\o}rdam}, \emph{Non-simple purely infinite C*-algebras},
 Amer. J. Math. \textbf{122} (2000), no.~3, 637--666.
 
 \bibitem{KR02}
 {\sc E. Kirchberg, M. R{\o}rdam}, \emph{Infinite non-simple C*-algebras: 
 absorbing the Cuntz algebra $\mathcal{O}_{\infty}$}, Adv. Math. 
 \textbf{167} (2002), no.~2, 195--264. 
 
 
  
\bibitem{Kishimoto}
{\sc A. Kishimoto,} {\it Outer automorphisms and reduced crossed products of simple C*-algebras,}
 Comm. Math. Phys. {\bf 81}  (1981), no.~3, 429--435.

\bibitem{Ki95}
{\sc A. Kishimoto,}
\emph{The Rohlin property for automorphisms of UHF~algebras},
 J. Reine Angew. Math. \textbf{465} (1995), 183--196. 

\bibitem{Ki96}
{\sc A. Kishimoto,} {\it A Rohlin property for one-parameter automorphism groups},
 Comm. Math. Phys. \textbf{179} (1996), no.~3, 599--622. 


 \bibitem{Lin.book}
 {\sc H. Lin,} {\it An introduction to the classification of amenable C*-algebras,}
World Scientific Publishing Co., Inc., River Edge, NJ, 2001.


\bibitem{Lin01}
{\sc H. Lin}, \emph{The tracial topological rank of C*-algebras}, 
 Proc. London Math. Soc. (3) \textbf{83} (2001), no.~1, 199--234. 
 
\bibitem{Lin01c}
{\sc H. Lin}, \emph{Tracially AF C*-algebras}, 
 Trans.  Amer. Math. Soc. \textbf{353} (2001), no.~2, 693--722.
 
 



 \bibitem{MS12}
{\sc H. Matui,  Y. Sato }, 
\emph{$\mathcal{Z}$-stability of crossed products by strongly outer actions},
 Comm. Math. Phys. \textbf{314} (2012), no.~1, 193--228. 

 \bibitem{Nawata}
{\sc N. Nawata}, \emph{Finite group actions on certain stably projectionless C*-algebras with the Rohlin property}, 
 Trans.  Amer. Math. Soc. \textbf{368} (2016), no.~1, 471--493.
 

\bibitem{ORT11}
{\sc E. Ortega, M.~R{\o}rdam, H. Thiel}, 
\emph{The Cuntz semigroup and comparison of open
projections}, J. Funct. Anal. {\bf 260} (2011), no.~12,  (2011) 3474--3493.

\bibitem{OP06}
{\sc H. Osaka, N. C. Phillips}, 
\emph{Stable and real rank for crossed products by
automorphisms with the tracial Rokhlin property},
Ergod. Theory Dyn. Syst. \textbf{26} (2006), no.~5, 1579--1621.

\bibitem{OP12}
{\sc H. Osaka, N. C. Phillips}, 
\emph{Crossed products by finite group actions with the Rokhlin property},
Math. Z. \textbf{ 270} (2012), no.~1--2, 19--42.


 \bibitem{Ph06}
{\sc N. C. Phillips}, 
\emph{Every simple higher dimensional noncommutative torus is an AT~algebra}, 
preprint 2006 (arXiv: 0609783  [math.OA]).

\bibitem{Ph.free}
{\sc N. C. Phillips}, \emph{Freeness of  actions of finite groups on C*-algebras}, 
Operator structures and dynamical systems,  217--257, Contemp. Math., 503, Amer. Math. Soc., 
Providence, RI, 2009.

\bibitem{Ph11}
{\sc N. C. Phillips}, \emph{The tracial Rokhlin property for actions of finite groups on C*-algebras},
Amer. J. Math. \textbf{133} (2011), no.~3, 581--636.

\bibitem{Ph12}
{\sc N. C. Phillips}, 
\emph{The tracial Rokhlin property is generic}, preprint 2012 (arXiv: 1209.3859 [math.OA]).

\bibitem{Ph14} 
{\sc N. C. Phillips}, \emph{Large subalgebras}, preprint 2014 (arXiv: 1408.5546v1 [math.OA]).

\bibitem{Ph15}
{\sc N. C. Phillips}, \emph{Finite cyclic group actions with the tracial Rokhlin property},
Trans. Amer. Math. Soc. \textbf{367} (2015), no.~8, 5271--5300. 



\bibitem{Ro02}
{\sc M.~R{\o}rdam}, {\emph{Classification of nuclear, simple C*-algebras}},
pages 1--145 of:
M.~R{\o}rdam and E.~St{\o}rmer,
{\emph{Classification of nuclear C*-algebras. Entropy in operator
algebras}}, Encyclopaedia of Mathematical Sciences vol.\  126,
Springer-Verlag, Berlin, 2002.


\bibitem{Rosenberg}
{\sc J. Rosenberg}, \emph{Appendix to O. Bratteli's paper on ``Crossed products of UHF-algebras"},
Duke Math. J. \textbf{46} (1979), no.~10, 25-26.

\bibitem{Sa15} 
{\sc L. Santiago}, \emph{Crossed products by actions of finite groups with the Rokhlin property},
Internat. J. Math. \textbf{26} (2015), no.~7, 1550042, 31 pp.
 

\bibitem{TW07} 
{\sc A. S. Toms, W. Winter}, \emph{Strongly self-absorbing C*-algebras},
 Trans. Amer. Math. Soc. \textbf{359} (2007), no.~8, 3999--4029.

\bibitem{Wang13}
{\sc Q. Wang}, \emph{Tracial Rokhlin property and non-commutative dimensions}, 
Thesis (Ph.D.)-Washington University in St. Louis. 2013. 

\bibitem{Wang18}
{\sc Q. Wang}, \emph{The tracial Rokhlin property for actions of amenable groups on
C*-algebras},  Rocky Mountain J. Math. \textbf{48} (2018), no.~4, 1307--1344.

\bibitem{WZ08} 
{\sc  W. Winter, J. Zacharias}, \emph{Completely positive maps of order zero},
M\"{u}nster J. Math.  \textbf{2} (2009), 311--324.
\end{thebibliography}




\end{document}